\documentclass[11pt]{amsart}
\usepackage[english]{babel}
\usepackage{amssymb}
\usepackage{amsthm}

\makeatletter
\theoremstyle{plain}
 \newtheorem{thm}{\protect\theoremname}[section]
  \theoremstyle{definition}
  \newtheorem{defn}[thm]{\protect\definitionname}
  \theoremstyle{plain}
  \newtheorem{lem}[thm]{\protect\lemmaname}
  \theoremstyle{plain}
  \newtheorem{question}{\protect\questionname}
  \theoremstyle{plain}
  \newtheorem{cor}[thm]{\protect\corollaryname}
  \theoremstyle{remark}
  \newtheorem{rem}[thm]{\protect\remarkname}
  \theoremstyle{remark}
  \newtheorem{claim}[thm]{\protect\claimname}

\DeclareMathOperator{\acc}{acc}
\DeclareMathOperator{\dom}{dom}
\DeclareMathOperator{\otp}{otp}
\DeclareMathOperator{\cf}{cf}

\DeclareMathOperator{\Refl}{\mathrm{Refl}}
\DeclareMathOperator{\Col}{Col}
\newcommand{\ZFC}{{\rm ZFC}}
\newcommand{\GCH}{{\rm GCH}}
\newcommand{\PFA}{{\rm PFA}}

\makeatother

\providecommand{\claimname}{Claim}
\providecommand{\corollaryname}{Corollary}
\providecommand{\definitionname}{Definition}
\providecommand{\lemmaname}{Lemma}
\providecommand{\questionname}{Question}
\providecommand{\remarkname}{Remark}
\providecommand{\theoremname}{Theorem}

\begin{document}

\title{Simultaneous stationary reflection and square sequences}
\author{Yair Hayut}
\address{School of Mathematical Sciences, Tel Aviv University\\ 
Tel Aviv, 69978, Israel}
\email{yair.hayut@math.huji.ac.il}
\author{Chris Lambie-Hanson}
\address{Department of Mathematics, Bar-Ilan University \\ 
Ramat Gan, 5290002, Israel}
\email{lambiec@macs.biu.ac.il}
\thanks{This research was undertaken while the second author was a Lady Davis 
  Postdoctoral Fellow. The author would like to thank the Lady Davis Fellowship 
Trust and the Hebrew University of Jerusalem.}
\keywords{square principles, stationary reflection, forcing, large cardinals}
\subjclass[2010]{03E05, 03E35, 03E55}
\begin{abstract}
We investigate the relationship between weak square principles and simultaneous reflection of stationary sets.
\end{abstract}
\maketitle

\section{Introduction}
The investigation of the tension between compactness and incompactness phenomena in set theory has been a fruitful line of research, touching on aspects of large cardinals, inner model theory, combinatorial set theory, and cardinal arithmetic. Prominent among incompactness principles in set theory are square principles. Square principles were first introduced by Jensen, and variations have been defined by a number of researchers since then. A notable weakening, which we study here, is due to Todorcevic. On the other hand, stationary reflection is an important compactness phenomenon. It was known early on that Jensen's original square principle is incompatible with stationary reflection. Subsequent work, particularly \cite{Magidor-Cummings-Foreman-Squares}, indicates that weakenings of Jensen's square principle are compatible with certain forms of stationary reflection but preclude stronger forms of simultaneous stationary reflection. 
In \cite{FontanellaHayut}, Fontanella and the first author show that Todorcevic's square principle $\square(\kappa)$ is compatible with a strong form of 
stationary reflection, known as \emph{Delta reflection}, at $\aleph_{\omega^2 + 1}$. In this paper, we continue this line of research by considering the relationship between a hierarchy 
of weakenings of Todorcevic's square principle and simultaneous stationary reflection. 

Our notation is for the most part standard. We use \cite{jech} as a standard reference 
for undefined terms. If $A$ is a set of ordinals, then $\otp(A)$ denotes the order type 
of $A$ and $\acc(A)$ denotes the set of accumulation points of $A$ below the strong supremum 
of $A$, i.e. the set $\{\beta \mid \beta < \sup(\{\alpha + 1 \mid \alpha \in A\})$ and $\beta = 
\sup(A \cap \beta)$\}. If $\kappa < \lambda$ are infinite cardinals, with $\kappa$ regular, then 
$S^\lambda_\kappa = \{\alpha < \lambda \mid \cf(\alpha) = \kappa\}$. $S^\lambda_{<\kappa}$, 
$S^\lambda_{\geq \kappa}$, and similar expressions are interpreted in the obvious way.

We start by recalling the definition of Jensen's original square principle.

\begin{defn}
\cite{Jensen-FineStructure} Let $\kappa$ be a cardinal. We say that $\langle C_{\alpha}\mid\alpha<\kappa^{+}\rangle$ is a $\square_{\kappa}$-sequence if 
the following hold.
\begin{enumerate}
\item For all $\alpha < \kappa^+$, $C_\alpha$ is a club in $\alpha$.
\item (Coherence) For all $\beta < \kappa^+$ and all $\alpha\in\acc(C_{\beta})$, $C_{\beta}\cap\alpha=C_{\alpha}$. 
\item For all $\alpha < \kappa^+$, $\otp(C_{\alpha})\leq\kappa$.
\end{enumerate}
$\square_\kappa$ is the assertion that a $\square_\kappa$-sequence exists. We say that a sequence is coherent if it satisfies $(1)$ and $(2)$. In this case we say that $\kappa^{+}$ is the 
\emph{length} of the sequence.
\end{defn}

\begin{rem}
For our purposes, we may assume that, for every $\alpha < \kappa^+$, $C_{\alpha+1}=\{\alpha\}$.
\end{rem}

We can think of $\square_{\kappa}$ as a coherent way of witnessing the singularity of every ordinal between $\kappa$ and $\kappa^{+}$. This coherence allows us to build various objects by induction. It also witnesses the incompactness of $\kappa^{+}$, since it is impossible to extend the sequence to a coherent sequence of length $\kappa^{+}+1$ without collapsing $\kappa^+$. From $\square_\kappa$, one can obtain more natural incompactness phenomena, such as the existence of a non-free Abelian group of cardinality $\kappa^+$ such that all smaller subgroups are free \cite{MagidorShelah}, 
the existence of a non-metrizable first-countable topological space such that all smaller subspaces are metrizable (see \cite{magidor_ast}), and others.  

Square sequences appear naturally in core models. In fact, the square principle was isolated from the investigation of the structure of the constructible universe. 
\begin{thm}
\cite{Jensen-FineStructure} Assume $V=L$. For every cardinal $\kappa$, $\square_{\kappa}$ holds.
\end{thm}
This result was widely generalized, and the best known result today is the following, due to Zeman.
\begin{thm}
\cite{ZemanSquares} Assume $V=K$, where $K$ is the Mitchell-Steel core model. Then, for every cardinal $\kappa$, $\square_{\kappa}$ holds.
\end{thm}

These results can be shown to hold for a various class of models of the form $L[E]$ under some standard assumptions (see \cite{schimmerling-zeman}), and thus they are expected to be true also in inner models for any large cardinal assumption below subcompact. 

In these results, the square sequences are global, which means that one can define a coherent sequence over all the singular ordinals and derive from it the square sequences for each cardinal $\kappa$. The existence of a global square sequence in the core model implies that, in order to get a model in which $\square_\kappa$ fails, we need a model in which the set of $K$-singular ordinals below $(\kappa^{+})^{V}$ does not contain a club, or, in other words, that $(\kappa^{+})^{V}$ is a Mahlo cardinal in $K$. Solovay proved that, indeed, for regular $\kappa$, this is the exact large cardinal axiom needed (see \cite{kanamori-sets-and-extensions}). On the other hand the situation at successors of singular cardinals is more complex. In the absence of Woodin cardinals, the Weak Covering Lemma, which states that, for every singular cardinal $\kappa$, $(\kappa^{+})^{V}=(\kappa^{+})^{K}$, holds, and therefore $\square_{\kappa}$ holds for every singular $\kappa$. Some stronger results appear in \cite{sargsyan2014nontame}. The upper bound for the failure of $\square_{\kappa}$ for singular $\kappa$ is a measurable subcompact cardinal. The notion of subcompact was defined by Jensen as a weakening of a supercompact cardinal.

Square principles have many combinatorial consequences; the one that will be most relevant for us is the failure of stationary reflection. 
\begin{defn}
Let $\lambda$ be a regular cardinal.
\begin{enumerate}
\item{Suppose $S \subseteq \lambda$ is a stationary set and $\alpha < \lambda$. $S$ \emph{reflects at $\alpha$} if $\cf(\alpha) > \omega$ and $S\cap\alpha$ is stationary in $\alpha$. $S$ \emph{reflects} if there is $\alpha < \lambda$ such that $S$ reflects at $\alpha$.}
\item{Suppose $\mathcal{S}$ be a collection of stationary subsets of $\lambda$ and $\alpha < \lambda$. $\mathcal{S}$ \emph{reflects simultaneously at $\alpha$} if, for all $S\in\mathcal{S}$, $S$ reflects at $\alpha$. $\mathcal{S}$ \emph{reflects simultaneously} if there is $\alpha < \lambda$ such that $\mathcal{S}$ reflects simultaneously at $\alpha$. If $S_0$ and $S_1$ are stationary sets, then we say $S_0$ and $S_1$ reflect simultaneously if $\{S_0, S_1\}$ reflects simultaneously.}
\item{Suppose $S \subseteq \lambda$ is stationary and $\kappa$ is a cardinal. $\Refl(< \kappa, S)$ is the statement asserting that, whenever $\mathcal{S}$ is a collection of stationary subsets of $S$ and $|\mathcal{S}| < \kappa$, $\mathcal{S}$ reflects simultaneously. $\Refl(< \kappa^+, S)$ will typically be written as $\Refl(\kappa, S)$, and $\Refl(1, S)$ will be written as $\Refl(S)$.}
\end{enumerate}
\end{defn}
Note that, if $S \subseteq \lambda$ does not reflect at stationarily many ordinals below $\lambda$, then there is a club $C \subseteq \lambda$ such that $C \cap S$ does not reflect. Thus, if $\Refl(< \kappa, S)$ holds, then every collection $\mathcal{S}$ of stationary subsets of $S$ with $|\mathcal{S}| < \kappa$ reflects simultaneously at stationarily many $\alpha < \lambda$.

The following well known theorem demonstrates the connection between square principles and stationary reflection (see \cite{Magidor-Cummings-Foreman-Squares}).
\begin{thm}
\label{thm:non refl square_kappa} Assume $\square_{\kappa}$. Then $\Refl(S)$ fails for every stationary $S \subseteq \kappa^+$.
\end{thm}
\begin{proof}
Let $\langle C_{\alpha}\mid\alpha<\kappa^{+}\rangle$ be a $\square_\kappa$-sequence and fix a stationary $S \subseteq \kappa^+$. Using Fodor's Lemma, find a stationary $S^{\prime}\subseteq S$ and $\delta\leq\kappa$ such that, for every $\alpha\in S^{\prime}$, $\otp (C_{\alpha})=\delta$. For every $\beta < \kappa^+$, $S^\prime \cap \acc(C_{\beta})$ contains at most one point, since if $\alpha\in S^\prime \cap \acc(C_{\beta})$, then $C_{\beta}\cap\alpha=C_{\alpha}$, and therefore $\otp(C_{\beta}\cap\alpha)=\delta$. Thus, $S'$ does not reflect, so $\Refl(S)$ fails.
\end{proof}
The following generalization of Jensen's square principle is due to Schimmerling \cite{SchimmerlingWeakerSquares}:
\begin{defn}
Let $\kappa$ and $\eta$ be cardinals. A sequence $\mathcal{C}=\langle\mathcal{C}_{\alpha}\mid\alpha<\kappa^{+}\rangle$ is a $\square_{\kappa,<\eta}$-sequence if: 
\begin{enumerate}
\item For all $\alpha < \kappa^+$, $\mathcal{C}_{\alpha}$ is a non-empty set of clubs in $\alpha$ and $|\mathcal{C}_\alpha| < \eta$. 
\item For all $\beta < \kappa^+$, $C \in \mathcal{C}_\beta$, and $\alpha\in\acc(C)$, $C\cap\alpha\in\mathcal{C}_{\alpha}$.
\item For all $\alpha < \kappa^+$ and $C\in\mathcal{C}_{\alpha}$, $\otp(C)\leq\kappa$.
\end{enumerate}
$\square_{\kappa, <\eta}$ is the assertion that there is a $\square_{\kappa, < \eta}$-sequence. We say that $\mathcal{C}$ is a coherent sequence of width $<\eta$ and length $\kappa^+$ if it satisfies conditions $1$ and $2$.
\end{defn}
This definition provides us a strict hierarchy of combinatorial principles. $\square_{\kappa,<\eta}\implies\square_{\kappa,<\eta^{\prime}}$ for every $1\leq\eta<\eta^{\prime}$, but it is consistent, relative to large cardinals, that $\neg\square_{\lambda,<\eta}\wedge\square_{\lambda,<\eta^{\prime}}$ holds (See Jensen \cite{jensen_separating_squares} for successors of regular cardinals, and Cummings, Foreman, and Magidor \cite{Magidor-Cummings-Foreman-Squares} for
successors of singulars). We denote by $\square_{\lambda,\eta}$ the principle $\square_{\lambda,<\eta^{+}}$, so $\square_{\lambda}$ is the same as $\square_{\lambda,1}$.

The principle $\square_{\lambda,\lambda}$ is called \emph{weak square} and is equivalent to the existence of a special $\lambda^+$-Aronszajn tree. The principle $\square_{\lambda,\lambda^{+}}$ is called \emph{silly square} and is provable in ZFC (see \cite[Chapter 2]{shelah1994cardinal-arithmetic}).

These weaker square principles also have an impact on stationary reflection. The following theorems are from \cite{Magidor-Cummings-Foreman-Squares}.
\begin{thm}
If $\kappa$ is an infinite cardinal and $\square_{\kappa,<\cf(\kappa)}$ holds, then $\Refl(S)$ fails for every stationary $S \subseteq \kappa^+$. 
\end{thm}

\begin{thm}
If $\cf(\kappa)=\omega$ and $\square_{\kappa,<\kappa}$ holds, then $\Refl(\aleph_0, S)$ fails for every stationary $S \subseteq \kappa^+$.
\end{thm}
Another way of weakening the definition of square, due to Todorcevic, is to replace condition (3), the restriction on order types, with its non-compactness consequence:
\begin{defn}
\cite{TodorcevicHandbook} Let $\lambda$ be a regular cardinal. A coherent sequence $\mathcal{C}=\langle C_{\alpha}\mid\alpha<\lambda\rangle$ is a $\square(\lambda)$-sequence if there is no club $D\subset\lambda$ such that, for every $\alpha\in\acc(D)$, $D\cap\alpha=C_{\alpha}$. 
\end{defn}
A club $D$ such that $D\cap\alpha=C_{\alpha}$ for every $\alpha$ is called a \emph{thread} through $\mathcal{C}$.

Note that this definition can be understood as asserting that there is no way to extend $\mathcal{C}$ to a coherent sequence of length $\lambda+1$. Unlike with $\square_{\kappa}$, it is consistent that one can enlarge the universe and add a thread to some $\square(\lambda)$-sequence without changing the cofinality of $\lambda$.

As in the case for $\square_{\kappa}$, we are interested in weaker versions of this principle. 
\begin{defn}
Let $\lambda$ be a regular cardinal, and let $\eta \leq \lambda$. $\mathcal{C}=\langle\mathcal{C}_{\alpha}\mid\alpha<\lambda\rangle$ is a $\square(\lambda,<\eta)$-sequence if it is a coherent sequence of width $<\eta$ and there is no club $D\subset\lambda$ such that, for every $\alpha\in\acc(D)$, $D\cap\alpha\in\mathcal{C}_{\alpha}$.
\end{defn}

In this paper, we investigate the extent to which $\square(\lambda)$ and its weakenings place restrictions on simultaneous reflection. In Section \ref{sec: zfc results}, we present some generalizations of a folklore result that $\square(\lambda)$ implies the failure of $\Refl(2, S)$ for every stationary $S \subseteq \lambda$ to square sequences of larger width. In the rest of the paper, we 
prove consistency results showing that these generalizations are close to sharp. In Section \ref{sec: forcing preliminaries}, we introduce some of the forcing technology that will be used for this purpose. In Section \ref{sec:consistency results}, we produce models of $\ZFC$ in which square principles and some amount of simultaneous stationary reflection hold together. We conclude with some unresolved questions.  

\section{ZFC Results} \label{sec: zfc results}

In this section, we prove various results indicating that square principles place limitations on the extent of simultaneous stationary reflection. We first present a folklore result.

\begin{thm}
\label{thm:(Lambie-Hanson)-simul refl}
Suppose $\lambda$ is a regular, uncountable cardinal and $\square(\lambda)$ holds. Then every stationary set $S\subseteq\lambda$ can be partitioned into two stationary sets that do not reflect simultaneously. 
\end{thm}

A proof of Theorem \ref{thm:(Lambie-Hanson)-simul refl} can be found in \cite{Chris-SimulReflection}. A stronger version of this theorem, in which the stationary set is partitioned into $\lambda$ stationary sets such that no two reflect simultaneously, can be found in \cite{rinot-chain-conditions}. We will present a different proof of Theorem \ref{thm:(Lambie-Hanson)-simul refl} that can be modified to prove a generalization. We start with some useful definitions and lemmas, beginning with a result of Kurepa from \cite{Kurepa}.

\begin{lem} \label{kurepa_lemma}
  Suppose $\kappa < \lambda$ are cardinals, with $\lambda$ regular, and suppose that $T$ is a tree of height $\lambda$, all of whose levels have size less than $\kappa$. Then $T$ has a cofinal branch.
\end{lem}

\begin{defn}
Let $\mathcal{C} = \langle \mathcal{C}_\alpha \mid \alpha < \lambda \rangle$ be a coherent sequence of length $\lambda$ and any width. A club $E\subseteq\lambda$ is a \emph{weak thread} through $\mathcal{C}$ if, for every $\alpha\in\acc(E)$, there is a $C\in\mathcal{C}_{\alpha}$, such that $E\cap\alpha\subseteq C$. 
\end{defn}

\begin{lem} \label{lem:weak thread}
  Suppose $\lambda$ is a regular, uncountable cardinal, $\kappa < \lambda$, and $\mathcal{C} = \langle \mathcal{C}_\alpha \mid \alpha < \lambda \rangle$ is a coherent sequence of length $\lambda$ and width $<\kappa$. If $\mathcal{C}$ has a weak thread, then $\mathcal{C}$ has a thread.
\end{lem}

\begin{proof}
  Suppose $E$ is a weak thread through $\mathcal{C}$. We define the tree of attempts to construct a thread through $\mathcal{C}$ which contains $E$. Let $\{\gamma_{\alpha}\mid\alpha < \lambda\}$ be an increasing enumeration of $\acc(E)$. Let $T=\{C\mid$ for some $\alpha < \lambda,$ $C\in\mathcal{C}_{\gamma_{\alpha}}$ and $E\cap\gamma_{\alpha}\subseteq C\}$, ordered by end-extension. $T$ is then a tree of height $\lambda$, and, for $\alpha < \lambda$, the elements of the $\alpha^{\mathrm{th}}$ level of $T$ are exactly the members of $T\cap\mathcal{C}_{\gamma_{\alpha}}$. Thus, since $\mathcal{C}$ has width $<\kappa$, all levels of $T$ are of size $<\kappa$. Since $\kappa < \lambda$, Lemma \ref{kurepa_lemma} implies that $T$ has a cofinal branch. If $b$ is a cofinal branch through $T$, then $\bigcup b$ is a thread through $\mathcal{C}$.
\end{proof}

\begin{lem}\label{lem:very weak thread} 
  Suppose $\lambda$ is a regular, uncountable cardinal and $\mathcal{C} = \langle \mathcal{C}_\alpha \mid \alpha < \lambda \rangle$ is a coherent sequence of length $\lambda$. Suppose $T_{0}$ and $T_{1}$ are unbounded subsets of $\lambda$ such that, for every $\alpha\in T_{1}$, there is $C\in\mathcal{C}_{\alpha}$ such that $T_{0}\cap\alpha\subseteq C$. Then $\mathcal{C}$ has a weak thread.
\end{lem}

\begin{proof}
Let $E=\acc(T_{0})\cap\acc(T_{1})$. We claim that $E$ is a weak thread for $\mathcal{C}$. To see this, fix $\alpha\in\acc(E)$, and let $\beta = \min(T_1 \setminus (\alpha + 1))$. By our assumption, there is $C \in \mathcal{C}_\beta$ such that $T_0 \cap \beta \subseteq C$. In particular, $\acc(T_0) \cap \beta \subseteq \acc(C)$, which implies that $C \cap \alpha \in \mathcal{C}_\alpha$ and $E \cap \alpha \subseteq C \cap \alpha$, as desired.
\end{proof}

We thus obtain the following corollary, which we will often use to prove that certain coherent sequences have threads.

\begin{cor}\label{cor:unbounded_cover_implies_thread}
  Suppose $\lambda$ is a regular, uncountable cardinal, $\kappa < \lambda$, $\mathcal{C} = \langle \mathcal{C}_\alpha \mid \alpha < \lambda \rangle$ is a coherent sequence of length $\lambda$ and width $<\kappa$, and there is an unbounded $A \subseteq \lambda$ such that, for all $\alpha \in A$, there is $C \in \mathcal{C}_\alpha$ such that $A \cap \alpha \subseteq C$. Then $\mathcal{C}$ has a thread.
\end{cor}

We are now ready to prove Theorem \ref{thm:(Lambie-Hanson)-simul refl},

\begin{proof}[Proof of Theorem \ref{thm:(Lambie-Hanson)-simul refl}]
Let $\mathcal{C}=\langle C_{\alpha}\mid\alpha<\lambda\rangle$ be a $\square(\lambda)$-sequence, and let $S \subseteq \lambda$ be a stationary set. We start by splitting $\lambda$ into two sets according to the behaviour of $\mathcal{C}$. Let $T_{bd}=\{\alpha < \lambda \mid\otp(C_{\alpha})<\alpha\}$ and $T_{ubd}=\{\alpha < \lambda \mid\otp(C_\alpha)=\alpha\}$. 

Suppose first that $S \cap T_{bd}$ is stationary. In this case, by Fodor's Lemma, we can find a stationary $S' \subseteq S \cap T_{bd}$ and a fixed ordinal $\delta < \lambda$ such that, for all $\alpha \in S'$, $\otp(C_\alpha) = \delta$. Then, by the argument in the proof of Theorem \ref{thm:non refl square_kappa}, $S'$ does not reflect. Thus, if $S = S_0 \ \dot{\cup} \  S_1$ is any partition of $S$ into stationary sets with $S_0 \subseteq S'$, then $S_0$ and $S_1$ do not reflect simultaneously.

Thus, we may assume that $S \cap T_{bd}$ is non-stationary. Let $D \subseteq \lambda$ be a club such that $D \cap S \cap T_{bd} = \emptyset$. For all $\beta < \lambda$, define a function $p_\beta : (D \cap S) \setminus (\beta + 1) \rightarrow \lambda$ by letting $p_\beta(\alpha)$ be the $\beta^{\mathrm{th}}$ member of $C_\alpha$, i.e.\ the unique $\gamma \in C_\alpha$ such that $\otp(C_\alpha \cap \gamma) = \beta$.

\begin{claim}
There are $\beta, \gamma < \lambda$ such that $p_\beta^{-1}(\gamma)$ and $S \setminus p_\beta^{-1}(\gamma)$ are both stationary.
\end{claim}
\begin{proof}
Suppose not. Then, for every $\beta < \lambda$, fix a club $E_\beta$ and an ordinal $\gamma_\beta < \lambda$ such that $p_\beta^{-1}(\gamma_\beta) \supseteq E_\beta \cap S$. If $\beta_0 < \beta_1 < \lambda$ are limit ordinals and $\alpha \in E_{\beta_0} \cap E_{\beta_1} \cap S$, then $\gamma_{\beta_0}, \gamma_{\beta_1} \in \acc(C_\alpha)$, so $C_{\gamma_{\beta_1}} \cap \gamma_{\beta_0} = C_{\gamma_{\beta_0}}$. Then $\bigcup \{C_{\gamma_\beta} \mid \beta < \lambda$, $\beta$ limit$\}$ is a thread through $\mathcal{C}$, which is a contradiction.
\end{proof}

Fix such a $\beta, \gamma < \lambda$, and let $S_0 = p_\beta^{-1}(\gamma)$ and $S_1 = S \setminus S_0$. We claim that $S_0$ and $S_1$ do not reflect simultaneously. To see this, fix $\alpha < \lambda$ of uncountable cofinality. If $\otp(C_\alpha) \leq \beta$ or $p_\beta(\alpha) \neq \gamma$, then $\acc(C_\alpha) \cap S_0 = \emptyset$. If $p_\beta(\alpha) = \gamma$, then $\acc(C_\alpha) \setminus (\gamma + 1) \cap S_1 = \emptyset$. In either case, $S_0$ and $S_1$ do not reflect simultaneously at $\alpha$.
\end{proof}

\subsection{Finite width}
We now generalize Theorem \ref{thm:(Lambie-Hanson)-simul refl} to square sequences of larger width. We start by considering sequences of finite width.

\begin{thm}\label{thm:non-reflecting-from-square} 
Suppose $\lambda$ is a regular, uncountable cardinal and $\square(\lambda,<\omega)$ holds. Then, for all stationary $S \subseteq \lambda$, $\Refl(2, S)$ fails.
\end{thm}
\begin{proof}
Let $\mathcal{C} = \langle \mathcal{C}_\alpha \mid \alpha < \lambda \rangle$ be a $\square(\lambda,<\omega)$-sequence. For each $\alpha < \lambda$, we let $\mathcal{C}_\alpha = \{C_{\alpha, i} \mid i < n_\alpha \}$ and, for convenience, we assume that the clubs are enumerated so that, for all $i_0 < i_1 < n_\alpha$, $\otp (C_{\alpha,i_0}) \leq \otp(C_{\alpha,i_1})$. For all $\alpha \in S$, let $m_\alpha \leq n_\alpha$ be least such that either $\otp(C_{\alpha, m_\alpha}) = \alpha$ or $m_\alpha = n_\alpha$. By shrinking $S$ if necessary, we may assume that there are $m \leq n < \omega$ and $\langle \eta_i \mid i < m \rangle$ such that, for all $\alpha \in S$, $m_\alpha = m$, $n_\alpha = n$, and, for all $i < m$, $\otp(C_{\alpha, i}) = \eta_i$. We may also assume that $S$ consists only of indecomposable ordinals, i.e.\ $\alpha$ such that, for all $\beta, \delta < \alpha$, we have $\beta + \delta < \alpha$.

\textbf{Case 1: $m=n$.} In this case, $S$ itself does not reflect. To see this, fix $\delta < \lambda$ with $\cf(\delta) > \omega$, and let $C \in \mathcal{C}_\delta$. If $\alpha \in S \cap \acc(C)$, then $C \cap \alpha \in \mathcal{C}_\alpha$, so $\otp (C \cap \alpha) \in \{\eta_i \mid i < m\}$. Thus, $|S \cap \acc(C)| \leq m$. In particular, a final segment of $\acc(C)$ is disjoint from $S$ and thus witnesses that $S$ does not reflect at $\delta$.

\textbf{Case 2: $m=0$.} For each $\alpha \in S$ and $i < n$, let $C_{\alpha, i}$ be enumerated in increasing order by $\{\xi^{\alpha, i}_\beta \mid \beta < \alpha\}$. For each $\beta < \lambda$, define a function $p_\beta$ on $S \setminus (\beta + 1)$ by letting $p_\beta = \{\xi^{\alpha, i}_\beta \mid i < n \}$.

\begin{lem}
Suppose that, for every $\beta < \lambda$, there is a club $E_\beta \subseteq \lambda$ and an ordinal $\gamma_\beta < \lambda$ such that $p_\beta(\alpha) \cap \gamma_\beta \neq \emptyset$ for every $\alpha \in E_\beta \cap S$. Then there is a club $F \subseteq \lambda$ such that, for every $\alpha \in S \cap \acc(F)$, $F \cap \alpha \subseteq \bigcup \mathcal{C}_\alpha$.
\end{lem}
\begin{proof}
Let $D$ be the club of all limit ordinals $\delta < \lambda$ such that, for all $\beta < \delta$, $\gamma_\beta < \delta$. Let $F = D \cap \triangle_{\beta < \lambda} E_\beta$. We claim that $F$ is as desired. To this end, let $\alpha \in S \cap \acc(F)$, and let $\delta \in F \cap \alpha$. Since $\alpha \in S \cap \bigcap_{\beta < \delta} E_\beta$, we have that, for all $\beta < \delta$, there is $i_\beta < n$ such that $\beta \leq \xi^{\alpha, i_\beta}_\beta < \delta$. Fix $i < n$ such that $i_\beta = i$ for unboundedly many $\beta < \delta$. Then $\delta \in C_{\alpha, i}$, so $F \cap \alpha \subseteq \bigcup \mathcal{C}_\alpha$.
\end{proof}

\begin{lem}\label{lem:Thread from cover}
Suppose there is a club $F \subseteq \lambda$ such that, for every $\alpha \in S \cap \acc(F)$, $F \cap \alpha \subseteq \bigcup \mathcal{C}_\alpha$. Then $\mathcal{C}$ has a thread.
\end{lem}
\begin{proof}
For every $\alpha \in S \cap \acc(F)$, let $k_\alpha$ be the size of the smallest subset $\mathcal{C}'_\alpha \subseteq \mathcal{C}_\alpha$ that covers $F \cap \alpha$ up to a bounded error, i.e.\ such that there is $\beta_\alpha < \alpha$ such that $F \cap (\beta_\alpha, \alpha) \subseteq \bigcup \mathcal{C}'_\alpha$. By Fodor's Lemma, there is $k^* < \omega$, $\beta^* < \lambda$, and a stationary $T \subseteq S \cap \acc(F)$ such that, for all $\alpha \in T$, $k_\alpha = k^*$ and $\beta_\alpha = \beta^*$ (in particular, $\alpha > \beta^*$).

\begin{claim}
For all $\alpha \in T$ and all $C \in \mathcal{C}'_\alpha$, $T \cap \alpha \subseteq \acc(C)$.
\end{claim}

\begin{proof}
Fix $\alpha \in T$ and $\delta \in T \cap \alpha$, let $\mathcal{D}_\alpha = \{C \in \mathcal{C}'_\alpha \mid \delta \not\in \acc(C) \}$, and suppose for sake of contradiction that $\mathcal{D}_\alpha \neq \emptyset$. Fix $\beta$ such that $\beta^* \leq \beta < \delta$ and, for all $C \in \mathcal{D}_\alpha$, $C \cap (\beta, \delta) = \emptyset$. It must therefore be the case that $F \cap (\beta, \delta) \subseteq \bigcup (\mathcal{C}'_\alpha \setminus \mathcal{D}_\alpha)$. Let $\mathcal{C}''_\delta = \{C \cap \delta \mid C \in \mathcal{C}'_\alpha \setminus \mathcal{D}_\alpha\}$. Then $|\mathcal{C}''_\delta| < k^*$ and $F \cap (\beta, \delta) \subseteq \bigcup \mathcal{C}''_\delta$, contradicting the fact that $k_\delta = k^*$ (since $\delta \in T$).
\end{proof}
Thus, by Corollary \ref{cor:unbounded_cover_implies_thread} applied to $T$, $\mathcal{C}$ has a thread.
\end{proof}
Combining these two lemmas, we see that there must be an ordinal $\beta<\lambda$ such that, for every $\gamma < \lambda$, there are stationarily many $\alpha \in S$ such that $p_{\beta}(\alpha)\cap\gamma=\emptyset$. Fix such a $\beta$. By shrinking $S$ if necessary, we may assume $S \subseteq \lambda \setminus (\beta + 1)$. For a finite set of ordinals below $\lambda$, $v$, let $S_{v}=p_{\beta}^{-1}(v)$.
\begin{lem}
\label{lem:non intersecting values to non-reflecting sets} Let $\ell < \omega$ and suppose that, for all $k < \ell$, $v_k$ is a finite set of ordinals below $\lambda$. If $\bigcap_{k < \ell}v_k = \emptyset$, then $\mathcal{S} = \{S_{v_k} \mid k < \ell\}$ does not reflect simultaneously. 
\end{lem}
\begin{proof}
Let $\beta < \delta < \lambda$ with $\cf(\delta) > \omega$, let $C=\bigcap\mathcal{C}_\delta$, and let $p_\beta(\delta)=u$. Note that, for every $\gamma \in S \cap \acc(C)$, $p_\beta(\gamma)\supseteq u$. This is because $D\cap\gamma\in\mathcal{C}_\gamma$ for every $D\in\mathcal{C}_\delta$, and the $\beta^{\mathrm{th}}$ element of $D\cap\gamma$ is the same as the $\beta^{\mathrm{th}}$ element of $D$. Since $\bigcap v_{i}=\emptyset$, there is $k<\ell$ such that $v_{k}\not\supseteq u$, and this implies that $S_{v_{k}}$ does not reflect at $\delta$. 
\end{proof}

By Fodor's Lemma, we can find a stationary set $S_0 \subseteq S$ and a finite set $v_0 \subseteq \lambda$ such that $p_{\beta}``S_{0}=\{v_0\}$. Let $\gamma_0=\max(v_0)$. By our choice of $\beta$, there are stationarily many $\alpha \in S$ such that $p_\beta(\alpha) \cap (\gamma_0 + 1) = \emptyset$, so, through another application of Fodor's Lemma, we can find a stationary $S_1 \subseteq S$ and a finite $v_1 \subseteq \lambda$ such that $p_{\beta}`` S_1=\{v_1\}$ and $v_1\cap(\gamma_{0}+1)=\emptyset$. In particular, $v_1\cap v_0=\emptyset$, so we can apply Lemma \ref{lem:non intersecting values to non-reflecting sets} and conclude that $S_0$ and $S_1$ do not reflect simultaneously.

\textbf{Case 3: $0 < m < n$.} We will reduce this case to Case 2. Recall that $\langle \eta_i \mid i < m \rangle$ is such that, for all $\alpha \in S$ and all $i < m$, $\otp (C_{\alpha, i}) = \eta_i$. For each $\alpha < \lambda$ and each $C \in \mathcal{C}_\alpha$, define $C^*$ as follows. If $\otp (C) \leq \eta_0$, let $C^* = C$. Otherwise, let $i_C < m$ be greatest such that $\eta_{i_C} < \otp (C)$, let $\xi \in C$ be such that $\otp (C \cap \xi) = \eta_{i_C}$, and let $C^* = C \setminus (\xi + 1)$. Define $\mathcal{C}^* = \langle \mathcal{C}^*_\alpha \mid \alpha < \lambda \rangle$ by letting, for all $\alpha < \lambda$, $\mathcal{C}^*_\alpha = \{C^* \mid C \in \mathcal{C}_\alpha\}$. It is routine to check that $\mathcal{C}^*$ is a $\square(\lambda, < \omega)$-sequence. Suppose $\alpha \in S$ and $i < n$. If $i < m$, we have arranged that, for all $\alpha < \beta < \lambda$ and all $C^* \in \mathcal{C}^*_\beta$, $C^* \cap \alpha \neq C^*_{\alpha, i}$. If $m \leq i < n$, then, since $\alpha$ is indecomposable, $\otp (C^*_{\alpha, i}) = \alpha$. For $\alpha < \lambda$, define $\mathcal{D}_\alpha$ by letting $\mathcal{D}_\alpha = \mathcal{C}^*_\alpha$ for $\alpha \not\in S$ and $\mathcal{D}_\alpha = \{C^*_{\alpha, i} \mid m \leq i < n \}$ for $\alpha \in S$. Then $\mathcal{D} = \langle \mathcal{D}_\alpha \mid \alpha < \lambda \rangle$ is a $\square(\lambda, < \omega)$-sequence and, for all $\alpha \in S$ and all $D \in \mathcal{D}_\alpha$, $\otp (D) = \alpha$. We may thus proceed as in Case 2, using $\mathcal{D}$ in place of $\mathcal{C}$.
\end{proof}

The proof of Theorem \ref{thm:non-reflecting-from-square} uses the fact that the width of the square sequence is finite in two places. The first is in finding a value $v\in\kappa^{<\omega}$ such that $p_{\beta}^{-1}(v)$ is stationary, and the second is in the proof of Lemma \ref{lem:Thread from cover}. We will see later that the generalization of Lemma \ref{lem:Thread from cover} to coherent sequences of infinite width is false. In fact, we will see that, if $\kappa < \lambda$ are infinite, regular cardinals, then there is consistently a $\square(\lambda, \kappa)$-sequence $\mathcal{C} = \langle \mathcal{C}_\alpha \mid \alpha < \lambda \rangle$ such that, for every limit ordinal $\alpha < \lambda$, $\alpha = \bigcup \mathcal{C}_\alpha$.
\subsection{Infinite width}

We move on now to consider square sequences of possibly infinite width. We first show that the existence of such square sequences necessarily implies the failure of some simultaneous reflection for stationary sets consisting of ordinals of sufficiently high cofinality.

\begin{thm} \label{thm:high_cofinality_reflection}
Suppose $\kappa < \lambda$ are cardinals, $\lambda$ is regular, and $\square(\lambda, < \kappa)$ holds. Then $\Refl(< \kappa, S)$ fails for all stationary $S \subseteq S^\lambda_{\geq \kappa}$.
\end{thm}

\begin{proof}
Suppose for sake of contradiction that $S \subseteq S^\lambda_{\geq \kappa}$ is stationary and $\Refl(< \kappa, S)$ holds. Let $\mathcal{C} = \langle \mathcal{C}_\alpha \mid \alpha < \lambda \rangle$ be a $\square(\lambda, < \kappa)$-sequence. For all $\beta \in S$, let $D_\beta = \bigcap_{C \in \mathcal{C}_\beta} \acc(C)$. Since $S \subseteq S^\lambda_{\geq \kappa}$ and $\mathcal{C}$ is a $\square(\lambda, < \kappa)$-sequence, $D_\beta$ is club in $\beta$ for all $\beta \in S$. For all $\alpha < \lambda$, let $S_\alpha = \{\beta \in S \mid \alpha \in D_\beta\}$. Let $A = \{\alpha < \lambda \mid S_\alpha$ is stationary$\}$.

\begin{claim}
$A$ is unbounded in $\lambda$.
\end{claim}

\begin{proof}
Fix $\alpha_0 < \lambda$. For $\beta \in S \setminus (\alpha_0 + 1)$, let $f(\beta) = \min(D_\beta \setminus \alpha_0)$. $f$ is regressive, so, by Fodor's Lemma, we can find $\alpha$ and a stationary $S' \subseteq S$ such that $f(\beta) = \alpha$ for all $\beta \in S'$. Then $\alpha \in A \setminus \alpha_0$. Since $\alpha_0$ was arbitrary, $A$ is unbounded in $\lambda$.
\end{proof}

\begin{claim} \label{claim:small-covering}
Suppose $\gamma \in A$ and $X \in [A \cap \gamma]^{<\kappa}$. Then there is $C \in \mathcal{C}_\gamma$ such that $X \subseteq C$.
\end{claim}

\begin{proof}
By assumption, $\mathcal{S} = \{S_\alpha \mid \alpha \in X \cup \{\gamma\} \}$ reflects simultaneously. Suppose $\mathcal{S}$ reflects simultaneously at $\delta$. Fix $E \in \mathcal{C}_\delta$. For every $\alpha \in X \cup \{\gamma\}$, let $\beta_\alpha \in \acc(E) \cap S_\alpha$. Then, since $\alpha \in D_{\beta_\alpha}$ and $E \cap \beta_\alpha \in \mathcal{C}_{\beta_\alpha}$, we have $\alpha \in \acc(E)$. In particular, $E \cap \gamma \in \mathcal{C}_\gamma$ and $X \subseteq E \cap \gamma$.
\end{proof}

\begin{claim} \label{claim:large-covering}
Suppose $\gamma \in A$. Then there is $C \in \mathcal{C}_\gamma$ such that $A \cap \gamma \subseteq C$. 
\end{claim}

\begin{proof}
Suppose not. For every $C \in \mathcal{C}_\gamma$, find $\alpha_C \in (A \cap \gamma) \setminus C$. Let $X = \{\alpha_C \mid C \in \mathcal{C}_\gamma \}$. Then $X \in [A \cap \gamma]^{<\kappa}$, so, by Claim \ref{claim:small-covering}, there is $C \in \mathcal{C}_\gamma$ such that $X \subseteq C$. In particular, $\alpha_C \in C$, contradicting our assumption.
\end{proof}

By Corollary \ref{cor:unbounded_cover_implies_thread} applied to $A$, $\mathcal{C}$ has a thread, contradicting the assumption that $\mathcal{C}$ is a $\square(\lambda, < \kappa)$-sequence.
\end{proof}

We now introduce a dichotomy for square sequences and show that, in each case, we get further failure of simultaneous stationary reflection.

\begin{defn}
Suppose $\lambda$ is a regular, uncountable cardinal and $\mathcal{C} = \langle \mathcal{C}_\alpha \mid \alpha < \lambda \rangle$ is a coherent sequence of any width. Let $A_{\mathcal{C}}$ be the set of $\alpha < \lambda$ such that there is a club $D_\alpha \subseteq \lambda$ such that, for every $\beta \in D_\alpha$, $\alpha \in \bigcup_{C \in \mathcal{C}_\beta} \acc(C)$. $\mathcal{C}$ is $\emph{full}$ if $A_{\mathcal{C}}$ is unbounded in $\lambda$.
\end{defn}

\begin{thm} \label{fullReflectionThm}
Suppose $\kappa < \lambda$ are uncountable cardinals, with $\lambda$ regular, and suppose there is a full $\square(\lambda, < \kappa)$-sequence. Then $\Refl(< \kappa, S)$ fails for all stationary $S \subseteq \lambda$.
\end{thm}

\begin{proof}
Suppose $\mathcal{C} = \langle \mathrm{C}_\alpha \mid \alpha < \lambda \rangle$ is a full $\square(\lambda, < \kappa)$-sequence, and let $S \subseteq \lambda$ be stationary. Let $A_{\mathcal{C}}$ be as given in the definition of fullness. Find an unbounded $A \subseteq A_{\mathcal{C}}$ and a fixed $\mu < \kappa$ such that, for all $\alpha \in A$, $|\mathcal{C}_\alpha| = \mu$. For each $\alpha \in A$, let $\mathcal{C}_\alpha = \{C_{\alpha, i} \mid i < \mu\}$, and let $D_\alpha$ be club in $\lambda$ such that, for all $\beta \in D_\alpha$, $\alpha \in \bigcup_{C \in \mathcal{C}_\beta} \acc(C)$.

For all $\alpha \in A$ and $i < \mu$, let $S_{\alpha, i}$ be the set of $\beta \in S$ such that, for all $C \in \mathcal{C}_\beta$, $C \cap \alpha \neq C_{\alpha, i}$.

\begin{claim}
For all but boundedly many $\alpha \in A$, for all $i < \mu$, $S_{\alpha, i}$ is stationary.
\end{claim}

\begin{proof}
Suppose for sake of contradiction that there is an unbounded $B \subseteq A$ and, for all $\alpha \in B$, an $i_\alpha < \mu$ such that $S_{\alpha, i_\alpha}$ is non-stationary. For all $\alpha \in B$, fix a club $E_\alpha$ in $\lambda$ such that $E_\alpha \cap S_{\alpha, i_\alpha} = \emptyset$. Define a tree $T$ as follows. Elements of $T$ are clubs $C_{\alpha, i_\alpha}$ for $\alpha \in B$. Order $T$ by end-extension, i.e., for all $\alpha_0 < \alpha_1$, both in $B$, let $C_{\alpha_0, i_{\alpha_0}} <_T C_{\alpha_1, i_{\alpha_1}}$ iff $C_{\alpha_1, i_{\alpha_1}} \cap \alpha_0 = C_{\alpha_0, i_{\alpha_0}}$.

This definition clearly makes $T$ a tree, and $|T| = \lambda$. We claim that $T$ has no antichain of size $\kappa$. To see this, let $T'$ be a subset of $T$ of size $\kappa$. Suppose $T' = \{C_{\alpha_\xi, i_{\alpha_\xi}} \mid \xi < \kappa\}$. Find $\beta \in S \cap \bigcap_{\xi < \kappa} E_{\alpha_\xi}$. For all $\xi < \kappa$, $\beta \in S \setminus T_{\alpha_\xi, i_{\alpha_\xi}}$, so there is $C \in \mathcal{C}_\beta$ such that $C \cap \alpha_\xi = C_{\alpha_\xi, i_{\alpha_\xi}}$. Since $|\mathcal{C}_\beta| < \kappa$, there are $\xi < \zeta < \kappa$ and $C \in \mathcal{C}_\beta$ such that $C \cap \alpha_\xi = C_{\alpha_\xi, i_{\alpha_\xi}}$ and $C \cap \alpha_\zeta = C_{\alpha_\zeta, i_{\alpha_\zeta}}$. Assuming without loss of generality that $\alpha_\xi < \alpha_\zeta$, we then have $C_{\alpha_\zeta, i_{\alpha_\zeta}} \cap \alpha_\xi = C_{\alpha_\xi, i_{\alpha_\xi}}$, so $T'$ is not an antichain.

In particular, every level of $T$ has size $< \kappa$. $T$ is thus a tree of height $\lambda$ with levels of size $<\kappa$, so, by Lemma \ref{kurepa_lemma}, $T$ has a cofinal branch, $b$. But then $\bigcup b$ is a thread for $\mathcal{C}$, which is a contradiction.
\end{proof}

Fix $\alpha \in A$ such that, for all $i < \mu$, $S_{\alpha, i}$ is stationary. For $i < \mu$, let $S_i = S_{\alpha, i} \cap D_\alpha$. We claim that $\{S_i \mid i < \mu \}$ does not reflect simultaneously. To see this, suppose for sake of contradiction that $\{S_i \mid i < \mu\}$ reflects simultaneously at $\delta$. Since each $S_i$ is a subset of $D_\alpha$, we must have $\delta \in D_\alpha$. There is thus $C \in \mathcal{C}_\delta$ such that $\alpha \in \acc(C)$. Fix $i^* < \mu$ such that $C \cap \alpha = C_{\alpha, i^*}$. Let $D = \acc(C) \setminus (\alpha + 1)$. Then $D$ is club in $\delta$ and, for all $\beta \in D$, $C \cap \beta \in \mathcal{C}_\beta$, so $D \cap S_{i^*} = \emptyset$. In particular, $S_{i^*}$ does not reflect at $\delta$, which is a contradiction.
\end{proof}

\begin{thm}
Suppose $\kappa < \lambda$ are regular, uncountable cardinals and there is a $\square(\lambda, < \kappa)$-sequence that is not full. Then $\Refl(2, \lambda)$ fails.
\end{thm}

\begin{proof}
  Let $\mathcal{C} = \langle \mathcal{C}_\alpha \mid \alpha < \lambda \rangle$ be a $\square(\lambda, < \kappa)$-sequence that is not full. Then $A_{\mathcal{C}}$ is bounded in $\lambda$, so there is $\alpha' < \lambda$ such that, for all $\alpha' \leq \alpha < \lambda$, there is a stationary $S_\alpha \subseteq \lambda$ such that, for all $\beta \in S_\alpha$, $\alpha \not\in \bigcup_{C \in \mathcal{C}_\beta} \acc(C)$. For each $\alpha < \lambda$, let $T_\alpha = \{\beta < \lambda \mid \alpha \in \bigcap_{C \in \mathcal{C}_\beta} \acc(C)\}$. Let $B = \{\alpha < \lambda \mid T_\alpha$ is stationary$\}$.

\begin{claim}
$B$ is unbounded in $\lambda$.
\end{claim} 

\begin{proof}
Fix $\alpha_0 < \lambda$. We will find $\alpha \in B \setminus \alpha_0$. Note that, for all $\beta \in S^\lambda_{\geq \kappa}$, $\bigcap_{C \in \mathcal{C}_\beta} \acc(C)$ is club in $\beta$. Define a function $f:S^\lambda_{\geq \kappa} \setminus (\alpha_0 + 1) \rightarrow \lambda$ by letting $f(\beta) = \min((\bigcap_{C \in \mathcal{C}_\beta} \acc(C)) \setminus \alpha_0)$. $f$ is a regressive function on a stationary set, so there is a stationary $T$ and a fixed $\alpha$ such that, for all $\beta \in T$, $f(\beta) = \alpha$. Then $\alpha \in B \setminus \alpha_0$.
\end{proof}
Fix $\alpha \in B \setminus \alpha'$. We claim that $S_\alpha$ and $T_\alpha$ do not reflect simultaneously. To see this, fix $\delta < \lambda$ of uncountable cofinality, and fix $C \in \mathcal{C}_\delta$. Let $D = \acc(C) \setminus (\alpha + 1)$. If $\alpha \in \acc(C)$, then $D \cap S_\alpha = \emptyset$, so $S_\alpha$ does not reflect at $\delta$. If $\alpha \not\in \acc(C)$, then $D \cap T_\alpha = \emptyset$, so $T_\alpha$ does not reflect at $\delta$.
\end{proof}

Using the ideas of the previous proofs we can show that proper forcing cannot add a thread to any narrow square sequence.

\begin{cor}
Let $\kappa < \lambda$ be cardinals, with $\lambda$ uncountable and regular. A proper forcing cannot add a thread to a $\square(\lambda, <\kappa)$-sequence.
\end{cor}
\begin{proof}
Let $\mathcal{C} = \langle C_{\alpha, i} \mid \alpha < \lambda, i < \mu_\alpha < \kappa\rangle$ be a $\square(\lambda, < \kappa)$-sequence. For $\alpha < \lambda$ 
and $i < \mu_\alpha$, let $S_{\alpha, i} = \{\beta \in S^{\lambda}_{\omega} \mid \forall j < \mu_\beta,\, C_{\beta, j} \cap \alpha \neq C_{\alpha, i}\}$. By the proof 
of Theorem \ref{fullReflectionThm}, we know that, for all sufficiently large $\alpha < \lambda$, for all $i < \mu_\alpha$, $S_{\alpha, i}$ is stationary. 

Let $\mathbb{P}$ be a forcing notion and let us assume that $\mathbb{P}$ adds a thread through $\mathcal{C}$. In $V^{\mathbb{P}}$, let $D$ be this thread, and let $\alpha < \lambda$ be a sufficiently large accumulation point of $D$. Then there is $i$ such that $D\cap \alpha = C_{\alpha, i}$. $S_{\alpha, i} \cap \acc(D) \setminus (\alpha + 1) = \emptyset$, and therefore $S_{\alpha, i}$ is non-stationary in $V^{\mathbb{P}}$. In particular, $\mathbb{P}$ does not preserve stationary subsets of $S^\lambda_\omega$ and therefore is not proper.  
\end{proof}
\begin{rem}
A $\square(\lambda,<\lambda)$-sequence can be $\lambda$-c.c.\ and, in particular, its threading forcing (which will be introduced in Section \ref{sec: forcing preliminaries}) can preserve any stationary subset of $\lambda$. It can also be $\sigma$-closed (unless there is $\mu < \lambda$ such that $\mu^\omega \geq \lambda$), and in this case the threading forcing will be proper.
\end{rem}
\section{Forcing preliminaries} \label{sec: forcing preliminaries}
\subsection{Adding a square sequence}
In this subsection we will describe the standard forcing for adding a $\square(\lambda)$-sequence. This forcing notion (in a slightly different form) is due to Jensen, who used it in order to separate the combinatorial principles $\square_{\omega_1}$ and $\square_{\omega_1, 2}$. We will state but not prove the basic properties of the forcing notion and refer the reader to \cite{Chris-SimulReflection} for further information and discussion.

Throughout this subsection, $\lambda$ is an arbitrary uncountable regular cardinal such that $\lambda^{<\lambda} = \lambda$.

Let us define $\mathbb{S}(\lambda, 1)$ which is a forcing notion that adds a $\square(\lambda)$-sequence using bounded approximations.
\begin{defn}
Let $\mathbb{S}(\lambda, 1)$ be the following forcing notion. A condition $s\in\mathbb{S}(\lambda, 1)$ is a sequence of the form $s = \langle s_i \mid i \leq \gamma\rangle$ where:
\begin{enumerate}
\item $\gamma < \lambda$;
\item for all $\alpha \leq \gamma$, if $\alpha$ is a limit ordinal then $s_\alpha$ is a closed unbounded subset of $\alpha$;
\item for all $\alpha < \beta < \lambda$, if $\alpha\in \acc(s_\beta)$, then $s_\alpha = s_\beta \cap \alpha$.
\end{enumerate}
The elements of $\mathbb{S}(\lambda, 1)$ are ordered by end-extension.
\end{defn} 
By genericity arguments, if $S\subseteq\mathbb{S}(\lambda, 1)$ is a generic filter, then $\mathcal{C} = \bigcup S$ is a $\square(\lambda) = \square(\lambda, 1)$-sequence (thus the index $1$). In this case, we do not distinguish between $\mathcal{C}$ and $S$ and say that $\mathcal{C}$ is a generic $\square(\lambda)$-sequence.

\begin{defn}
Let $\mathcal{C}$ be a $\square(\lambda)$-sequence. Then $\mathbb{T}(\mathcal{C})$ is the threading forcing for $\mathcal{C}$. The elements of $\mathbb{T}(\mathcal{C})$ are the members of $\mathcal{C}$ and the order of the forcing is end-extension.
\end{defn} 

When $\mathcal{C}$ is clear from the context, we omit it. Let us remark that, if $\mathcal{C}$ is a generic $\square(\lambda)$-sequence, then $\mathbb{T}(\mathcal{C})$ is non-atomic and, if $T\subseteq \mathbb{T}(\mathcal{C})$ is a generic filter, then $\bigcup T$ is a thread through $\mathcal{C}$. In this case, the forcing $\mathbb{S}(\lambda, 1)\ast\mathbb{T}(\mathcal{C})$ contains a dense $\lambda$-directed closed subset.

\begin{defn}
  Suppose that $T$ is a stationary subset of $\lambda$. $\mathrm{CU}(T)$ is the standard poset 
  to shoot a club through $T$, i.e., the poset consisting of 
  closed, bounded subsets $c$ of $\lambda$ such that $c \subseteq T$, ordered by end-extension.
\end{defn}

\begin{defn}
  Let $\mathbb{T}$ be a forcing notion, and let $S\subseteq \lambda$. 
  \begin{enumerate}
    \item $S$ is $\mathbb{T}$-\emph{fragile} if $\Vdash_{\mathbb{T}} ``\check{S}$ is non-stationary$."$ 
  \end{enumerate}
\end{defn}

\begin{defn}
  Let $\mathbb{T}$ be a forcing notion. An \emph{iteration to kill $\mathbb{T}$-fragile subsets of 
$\lambda$} is a forcing iteration $\langle \mathbb{P}_\alpha, \dot{\mathbb{Q}}_\beta \mid 
\alpha \leq \delta, \beta < \delta \rangle$, 
  where $\delta$ is an ordinal, satisfying:
  \begin{enumerate}
    \item the iteration is taken with supports of size $<\lambda$;
    \item for every $\beta < \delta$, there is a $\mathbb{P}_\beta$-name $\dot{S}_\alpha$ for a 
      $\mathbb{T}$-fragile subset of $\lambda$ such that $\Vdash_{\mathbb{P}_\beta}``\dot{Q}_\beta 
      = \mathrm{CU}(\lambda \setminus \dot{S}_\alpha)."$
  \end{enumerate}
\end{defn}

\begin{lem}\label{lem: square+killing stat+thread is kappa closed}
  Let $\mathbb{S}$ be a forcing notion and let $\mathbb{T}$ be a forcing notion in $V^\mathbb{S}$. Assume that $\mathbb{S} \ast \dot{\mathbb{T}}$ contains a $\lambda$-directed closed dense subset.
  In $V^\mathbb{S}$, let $\mathbb{P}$ be an iteration to kill $\mathbb{T}$-fragile subsets of $\lambda$. Then $\mathbb{S}\ast\dot{\mathbb{P}}\ast\dot{\mathbb{T}}$ contains a $\lambda$-directed closed dense subset. Moreover, after forcing with $\mathbb{S}\ast\dot{\mathbb{T}}$, $\mathbb{P}$ contains a dense $\lambda$-directed closed subset.
\end{lem}

\begin{proof}
In $V^{\mathbb{S}}$, let $\langle \mathbb{P}_\alpha, \dot{\mathbb{Q}}_\beta \mid \alpha \leq \delta, \beta < \delta \rangle$ be an iteration 
to kill $\mathbb{T}$-fragile subsets of $\lambda$, with 
$\mathbb{P} = \mathbb{P}_\delta$. 
For $\beta < \delta$, let $\dot{S}_\beta$ be the $\mathbb{P}_\beta$-name for a $\mathbb{T}$-fragile subset of $\lambda$ used to define $\dot{\mathbb{Q}}_\beta$. Let $\dot{C}_\beta$ be a 
$\mathbb{P}_\beta \ast \mathbb{T}$-name for a club in $\lambda$ disjoint from $\dot{S}_\beta$. Working in $V$, we will abuse notation and interpret $\dot{S}_\beta$ and $\dot{C}_\beta$ as $\mathbb{S} \ast \dot{\mathbb{P}}_\beta$ and $\mathbb{S} \ast \dot{\mathbb{P}}_\beta \ast \dot{\mathbb{T}}$-names, respectively. Let $\mathbb{U}_0$ be the $\lambda$-directed closed subset of $\mathbb{S} \ast \dot{\mathbb{T}}$. Let $\mathbb{U}$ be the set of $(s, \dot{p}, \dot{t}) \in \mathbb{S} \ast \dot{\mathbb{P}} \ast \dot{\mathbb{T}}$ such that:
\begin{itemize}
    \item{$(s, \dot{t}) \in \mathbb{U}_0$;}
    \item{there is $a \in V$ such that $s \Vdash_{\mathbb{S}} ``\dom(\dot{p}) = \check{a}"$;}
    \item{for all $\alpha \in a$, there is a closed, bounded $c_\alpha \subset \lambda$ such that $(s, \dot{p} \restriction \alpha) \Vdash_{\mathbb{S} \ast \dot{\mathbb{P}}_\alpha} 
      ``\dot{p}(\alpha) = \check{c}_\alpha"$;}
    \item{for all $\alpha \in a$, $(s, \dot{p} \restriction \alpha, \dot{t}) \Vdash_{\mathbb{S} \ast \dot{\mathbb{P}}_\alpha \ast \dot{\mathbb{T}}} ``\max(\check{c}_\alpha) \in \dot{C}_\alpha."$}
  \end{itemize}
  The proof that $\mathbb{U}$ is a dense, $\lambda$-directed closed subset can be found in \cite{Chris-Reflection_2}. The final sentence in the statement of the Lemma 
  follows immediately.
\end{proof}

\begin{cor}\label{cor: square with no fragile sets}
There is a forcing notion $\mathbb{Q}$ such that:
\begin{enumerate}
  \item $|\mathbb{Q}| = 2^\lambda$ and $\mathbb{Q}$ is $\lambda^+$-c.c.
  \item There is $\mathbb{T}\in V^{\mathbb{Q}}$ of cardinality $\lambda$ such that $\mathbb{Q}\ast\dot{\mathbb{T}}$ contains a $\lambda$-directed closed dense subset.
\item $\Vdash_{\mathbb{Q}} ``\square(\lambda)."$
\item In $V^{\mathbb{Q}}$, there are no $\mathbb{T}$-fragile stationary subsets of $\lambda$.
\end{enumerate}
\end{cor}
\begin{proof}
  We apply Lemma~\ref{lem: square+killing stat+thread is kappa closed} for $\mathbb{S} = \mathbb{S}(\lambda, 1)$, 
  $\mathbb{T} = \mathbb{T}(\mathcal{C})$, (where $\mathcal{C}$ is the generic 
  $\square(\lambda)$-sequence added by $\mathbb{S}$), and $\mathbb{P} = 
  \langle \mathbb{P}_\alpha, \dot{\mathbb{R}}_\beta \mid \alpha \leq 2^\lambda, \beta < 2^\lambda \rangle$ an iteration 
  to kill $\mathbb{T}$-fragile subsets of $\lambda$. 
  By standard bookkeeping arguments, we can arrange so that, for every $\mathbb{S} * \dot{\mathbb{P}}$-name 
  $\dot{S}$ for a $\mathbb{T}$-fragile subset of $\lambda$, there is an $\alpha < 2^\lambda$ 
  and a $\mathbb{P}_\alpha$-name $\dot{S}_\alpha$ such that:
  \begin{itemize}
    \item $\Vdash_{\mathbb{P}_\alpha}``\dot{R}_\alpha = \mathrm{CU}(\lambda \setminus \dot{S}_\alpha)$;
    \item $\Vdash_{\mathbb{P}}``\dot{S}_\alpha = \dot{S}."$
  \end{itemize}
  In particular, there are no $\mathbb{T}$-fragile stationary subsets of $\lambda$ 
  in $V^{\mathbb{S} \ast \dot{\mathbb{P}}}$. Let $\mathbb{Q} = \mathbb{S} \ast \dot{\mathbb{P}}$.

  The only part of the corollary that requires further argument is (3). It suffices to show that the iteration $\mathbb{P}$ cannot add a thread through the $\square(\lambda)$-sequence $\mathcal{C}$ generically added by $\mathbb{S}$. By Lemma \ref{lem: square+killing stat+thread is kappa closed}, after forcing over $V^{\mathbb{S}}$ with $\mathbb{T}$, $\mathbb{P}$ has a $\lambda$-directed closed dense subset. In particular, forcing with $\mathbb{P}\times\mathbb{P}$ in $V^{\mathbb{S}}$ preserves the regularity of $\lambda$. But if $\mathbb{P}$ adds a thread through $\mathcal{C}$, then $\mathbb{P}\times\mathbb{P}$ will add two distinct threads to $\mathcal{C}$, thus forcing $\cf(\lambda) = \omega$. 
\end{proof}

The discussion above translates with minimal changes to the case of square sequences with arbitrary width. 

\begin{defn}
Let $\lambda$ be a regular cardinal, and let $\kappa \leq \lambda$. $\mathbb{S}(\lambda, <\kappa)$ is the forcing notion for adding a $\square(\lambda, <\kappa)$-sequence using bounded approximations. A condition $s\in\mathbb{S}(\lambda, <\kappa)$ is a sequence of the form $s = \langle s_i \mid i \leq \gamma\rangle$ where
\begin{enumerate}
\item $\gamma < \lambda$.
\item For all $\alpha \leq \gamma$, if $\alpha$ is a limit ordinal, then $s_\alpha$ is a non-empty set of closed, unbounded subsets of $\alpha$, and $|s_\alpha| < \kappa$.
\item For all $\alpha < \beta \leq \gamma$ and all $C\in s_\beta$, if $\alpha \in \acc(C)$, then $C\cap \alpha \in s_\alpha$.
\end{enumerate}
The elements of $\mathbb{S}(\lambda, <\kappa)$ are ordered by end-extension.
\end{defn} 
\begin{defn}
Let $\mathcal{C}$ be a $\square(\lambda, <\kappa)$-sequence. Then $\mathbb{T}(\mathcal{C})$ is the threading forcing for $\mathcal{C}$. The elements of $\mathbb{T}(\mathcal{C})$ are the members of $\bigcup \mathcal{C}$, and the order of the forcing is end-extension.
\end{defn}

For a forcing $\mathbb{P}$ and a cardinal $\nu$, let $\mathbb{P}^\nu$ denote the full-support product of $\nu$ copies of $\mathbb{P}$.

\begin{lem}
  For every $\rho < \kappa$ the forcing $\mathbb{S}(\lambda, <\kappa)\ast \dot{\mathbb{T}}(\mathcal{C})^{\rho}$ has a $\kappa$-directed closed dense subset.
\end{lem}
\begin{cor}\label{cor: wide square with no fragile sets}
Let $\lambda$ be a regular cardinal, and let $\kappa \leq \lambda$. There is a forcing notion $\mathbb{Q}$ such that:
\begin{enumerate}
  \item $|\mathbb{Q}| = 2^\lambda$ and $\mathbb{Q}$ is $\lambda^+$-c.c.
  \item There is $\mathbb{T}\in V^{\mathbb{Q}}$ of cardinality $\lambda$ such that $\mathbb{Q}\ast\dot{\mathbb{T}}^\rho$ has a $\lambda$-directed closed dense subset for all $\rho < \kappa$.
\item $\Vdash_{\mathbb{Q}} ``\square(\lambda, <\kappa)."$
\item In $V^{\mathbb{Q}}$, there are no $\mathbb{T}$-fragile stationary subsets of $\lambda$.
\end{enumerate}
\end{cor}

\subsection{Adding an indexed square sequence}\label{subsec: indexed square}

We now introduce an indexed strengthening of $\square(\lambda, \kappa)$.

\begin{defn}
Let $\kappa < \lambda$ be infinite regular cardinals. $\mathcal{C} = \langle C_{\alpha, i} \mid \alpha < \lambda, i(\alpha) \leq i < \kappa \rangle$ is a $\square^{\mathrm{ind}}(\lambda, \kappa)$-sequence if the following hold.
\begin{enumerate}
\item{For all $\alpha < \lambda$, $i(\alpha) < \kappa$.}
\item{For all limit $\alpha < \lambda$ and $i(\alpha) \leq i < \kappa$, $C_{\alpha, i}$ is club in $\alpha$.}
\item{For all limit $\alpha < \lambda$ and $i(\alpha) \leq i < j < \kappa$, $C_{\alpha, i} \subseteq C_{\alpha, j}$.}
\item{For all limit $\alpha < \beta < \lambda$ and $i(\beta) \leq i < \kappa$, if $\alpha \in \acc(C_{\beta, i})$, then $i(\alpha) \leq i$ and $C_{\beta, i} \cap \alpha = C_{\alpha, i}$.}
\item{For all limit $\alpha < \beta < \lambda$, there is $i(\beta) \leq i < \kappa$ such that $\alpha \in \acc(C_{\beta, i})$.}
\item{There is no club $D \subseteq \lambda$ such that, for all $\alpha \in \acc(D)$, there is $i(\alpha) \leq i < \kappa$ such that $D \cap \alpha = C_{\alpha, i}$.}
\end{enumerate}
$\square^{\mathrm{ind}}(\lambda, \kappa)$ holds if there is a $\square^{\mathrm{ind}}(\lambda, \kappa)$-sequence.
\end{defn}

\begin{rem}
$\square^{\mathrm{ind}}(\lambda, \kappa)$ is a generalization of the indexed square notion $\square^{\mathrm{ind}}_{\mu, \cf(\mu)}$ studied in \cite{Magidor-Cummings-Foreman-Squares} and \cite{cummings-schimmerling}. In the forcing constructions used in those papers to add $\square^{\mathrm{ind}}_{\mu, \kappa}$-sequences, it is important that $\mu$ is singular and $\kappa = \cf(\mu)$. Removing the order-type restriction and moving to $\square^{\mathrm{ind}}(\lambda, \kappa)$ gives us much more freedom with regards to the width and length of our indexed square sequences, and this freedom will be exploited in consistency results later in the paper.
\end{rem}

It is clear that a $\square^{\mathrm{ind}}(\lambda, \kappa)$-sequence is a full $\square(\lambda, \kappa)$-sequence and so, by Theorem \ref{fullReflectionThm}, 
$\square^{\mathrm{ind}}(\lambda, \kappa)$ implies the failure of $\Refl(\kappa, S)$ for every stationary $S \subseteq \lambda$. We will see in Section \ref{sec:consistency results} 
that this is sharp. The following is easily seen. A proof can be found in \cite{Chris-narrow_systems}.

\begin{lem}
Let $\kappa < \lambda$ be regular cardinals. The above definition of a $\square^{\mathrm{ind}}(\lambda, \kappa)$-sequence is unchanged if item (6) is replaced by the following seemingly weaker condition:
\[
\mbox{There is no club } D \subseteq \lambda \mbox{ and } i < \kappa \mbox{ such that, } \forall\alpha \in \acc(D), D \cap \alpha = C_{\alpha,i}.
\]
\end{lem}

We now define a forcing poset designed to add an indexed square sequence.

\begin{defn}
Suppose $\kappa < \lambda$ are infinite regular cardinals. Let $\mathbb{S}^{\mathrm{ind}}(\lambda, \kappa)$ be a forcing poset whose conditions are all $p = \langle C^p_{\alpha, i} \mid \alpha \leq \gamma^p, i(\alpha)^p \leq i < \kappa \rangle$ satisfying the following conditions.
\begin{enumerate}
\item{$\gamma^p < \lambda$ is a limit ordinal and, for all $\alpha \leq \gamma^p, i(\alpha)^p < \kappa$.}
\item{For all limit $\alpha \leq \gamma^p$ and all $i(\alpha)^p \leq i < \kappa$, $C^p_{\alpha, i}$ is a club in $\alpha$.}
\item{For all limit $\alpha \leq \gamma^p$ and all $i(\alpha)^p \leq i < j < \kappa$, $C^p_{\alpha, i} \subseteq C^p_{\alpha, j}$.}
\item{For all limit $\alpha < \beta \leq \gamma^p$ and all $i(\beta)^p \leq i < \kappa$, if $\alpha \in \acc(C^p_{\beta, i})$, then $i(\alpha)^p \leq i$ and $C^p_{\beta, i} \cap \alpha = C^p_{\alpha, i}$.}
\item{For all limit $\alpha < \beta \leq \gamma^p$, there is $i(\beta)^p \leq i < \kappa$ such that $\alpha \in \acc(C^p_{\beta, i})$.}
\end{enumerate}
If $p,q \in \mathbb{S}^{\mathrm{ind}}(\lambda, \kappa)$, then $q \leq p$ iff $q$ end-extends $p$.
\end{defn}

The following is proven in \cite{Chris-narrow_systems}.

\begin{lem}
Let $\kappa < \lambda$ be regular cardinals, and let $\mathbb{S} = \mathbb{S}^{\mathrm{ind}}(\lambda, \kappa)$.
\begin{enumerate}
\item{$\mathbb{S}$ is $\kappa$-directed closed.}
\item{$\mathbb{S}$ is $\lambda$-strategically closed.}
\item{If $G$ is $\mathbb{S}$-generic over $V$ and $\mathcal{C} = \bigcup G$, then $\mathcal{C}$ is a $\square^{\mathrm{ind}}(\lambda, \kappa)$-sequence in $V[G]$.}
\end{enumerate}
\end{lem}

We now introduce a family of forcings designed to thread an indexed square sequence.

\begin{defn}
  Let $\kappa < \lambda$ be regular cardinals, let $\mathcal{C} = \langle C_{\alpha, i} \mid \alpha < \lambda, i(\alpha) \leq i < \kappa \rangle$ be a $\square^{\mathrm{ind}}(\lambda, \kappa)$-sequence, and let $i < \kappa$. $\mathbb{T}_i(\mathcal{C})$ is the forcing poset whose conditions are all $C_{\alpha, i}$ such that $\alpha < \lambda$ is a limit ordinal and $i(\alpha) \leq i$. $\mathbb{T}_i(\mathcal{C})$ is ordered by end-extension. 
\end{defn}

\begin{lem}\label{lem: amalgamation for indexed square}
  Let $\kappa < \lambda$ be regular cardinals, and let $\mathbb{S} = \mathbb{S}^{\mathrm{ind}}(\lambda, \kappa)$. Let $\dot{\mathcal{C}} = \langle \dot{C}_{\alpha, i} \mid \alpha < \check{\lambda}, \dot{i(\alpha)} \leq i < \check{\kappa} \rangle$ be a canonical $\mathbb{S}$-name for the generically-introduced $\square^{\mathrm{ind}}(\lambda, \kappa)$-sequence and, for $i < \kappa$, let $\dot{\mathbb{T}}_i$ be an $\mathbb{S}$-name for $\mathbb{T}_i(\dot{\mathcal{C}})$.
\begin{enumerate}
\item{For all $i < \kappa$, $\mathbb{S} \ast \dot{\mathbb{T}}_i$ has a dense $\lambda$-directed closed subset.}
\item{Let $i < j < \kappa$ and, in $V^{\mathbb{S}}$, define $\pi_{i,j} : \mathbb{T}_i \rightarrow \mathbb{T}_j$ by letting, for all $C_{\alpha, i} \in \mathbb{T}_i$, $\pi_{i,j}(C_{\alpha, i}) = C_{\alpha, j}$. Then $\pi_{i,j}$ is a projection.}
\end{enumerate} 
\end{lem}

\begin{proof}
  We first establish (1). The proof is standard but included for completeness. Fix $i < \kappa$, and let $\mathbb{U}_i$ be the set of $(p, \dot{t}) \in \mathbb{S} \ast \dot{\mathbb{T}}_i$ such that $p \Vdash ``\dot{i}(\gamma^p) \leq i$ and $\dot{t} = \dot{C}_{\gamma^p, i}"$.
  We first show that $\mathbb{U}_i$ is dense. To this end, fix $(p_0, \dot{t}_0) \in \mathbb{S} \ast \dot{\mathbb{T}}_i$. By strengthening $p_0$ if necessary, we may assume that there is $\alpha < \lambda$ such that $p_0 \Vdash ``\dot{t}_0 = \dot{C}_{\alpha, i}"$ and that $\gamma^{p_0} \geq \alpha$. Let $\gamma = \gamma^{p_0} + \omega$. We will define $p \leq p_0$ with $\gamma^p = \gamma$. To do this, we need only specify $i(\gamma)^p$ and $C^p_{\gamma, j}$ for $i(\gamma)^p \leq j < \kappa$. Let $i(\gamma)^p = i$. Let $j^* < \kappa$ be least such that $\alpha \in \acc(C^{p_0}_{\gamma^p, j})$. If $i \leq j < j^*$, let $C^p_{\gamma, j} = C^{p_0}_{\alpha, j} \cup \{\alpha\} \cup \{\gamma^{p_0} + n \mid n < \omega \}$. If $i, j^* \leq j < \kappa$, let $C^p_{\gamma, j} = C^{p_0}_{\gamma^{p_0}, j} \cup \{\gamma^{p_0} + n \mid n < \omega \}$. Let $t = C^p_{\gamma, i}$, and let $\dot{t}$ be an $\mathbb{S}$-name forced to be equal to $t$. Then $(p, \dot{t}) \leq (p_0, \dot{t}_0)$, and $(p, \dot{t}) \in \mathbb{U}_i$.

  We next show that $\mathbb{U}_i$ is $\lambda$-directed closed. Note first that $\mathbb{U}_i$ is tree-like, i.e.\ if $u,v,w \in \mathbb{U}_i$ and $w \leq u, v$, then $u$ and $v$ are comparable. It thus suffices to show that $\mathbb{U}_i$ is $\lambda$-closed. Thus, let $\eta < \lambda$ be a limit ordinal, and let $\langle (p_\xi, \dot{t}_\xi) \mid \xi < \eta \rangle$ be a strictly decreasing sequence of conditions from $\mathbb{U}_i$. Let $\gamma = \sup(\{\gamma^{p_\xi} \mid \xi < \eta \})$. We first define $p \in \mathbb{S}$ so that $\gamma^p = \gamma$. It suffices to define $i(\gamma)^p$ and $C^p_{\gamma, j}$ for $i(\gamma)^p \leq j < \kappa$. Let $i(\gamma)^p = i$ and, for $i(\gamma)^p \leq j < \kappa$, let $C^p_{\gamma, j} = \bigcup_{\xi < \eta} C^{p_\xi}_{\gamma_\xi, j}$. Let $t = C^p_{\gamma, i}$, and let $\dot{t}$ be an $\mathbb{S}$-name forced to be equal to $t$. It is easily verified that $(p, \dot{t})$ is a lower bound for $\langle (p_\xi, \dot{t}_\xi) \mid \xi < \eta \rangle$ and that $(p, \dot{t}) \in \mathbb{U}_i$.

  We finally show (2). Let $G$ be $\mathbb{S}$-generic over $V$, and, in $V[G]$, let $\mathcal{C} = \langle C_{\alpha, i} \mid \alpha < \lambda, i(\alpha) \leq i < \kappa \rangle$ be the generically added $\square^{\mathrm{ind}}(\lambda, \kappa)$-sequence. Fix $i < j < \kappa$. It is clear that $\pi_{i,j}$ is order-preserving. It thus suffices to show that, for all $t_0 \in \mathbb{T}_i$ and $s \leq \pi_{i,j}(t_0)$ in $\mathbb{T}_j$, there is $t \leq t_0$ in $\mathbb{T}_i$ such that $\pi_{i,j}(t) \leq s$. Fix such a $t_0$ and $s$. Let $t_0 = C_{\gamma_0, i}$ and $s = C_{\gamma_1, j}$. By an easy density argument, we can find $\gamma_2 < \lambda$ such that $i(\gamma_2) \leq i$, $\gamma_0 \in \acc(C_{\gamma_2, i})$, and $\gamma_1 \in \acc(C_{\gamma_2, j})$. Let $t = C_{\gamma_2, i}$. Then $t \leq t_0$ and $\pi_{i,j}(t) \leq s$. 
\end{proof}

An argument similar to that in the proof of Lemma \ref{lem: square+killing stat+thread is kappa closed} yields the following Lemma.

\begin{lem} \label{indexed_iteration_lemma}
  Suppose $\kappa < \lambda$ are infinite, regular cardinals. Let $\mathbb{S}$ be $\mathbb{S}^{\mathrm{ind}}(\lambda, \kappa)$, and, in $V^\mathbb{S}$, let $\mathcal{C}$ 
  be the generically-added $\square^{\mathrm{ind}}(\lambda, \kappa)$-sequence. In $V^\mathbb{S}$ let $\mathbb{T} = \bigoplus_{i < \kappa} \mathbb{T}_i$ (the lottery sum of all the threading forcings $\mathbb{T}_i(\mathcal{C})$ for $i < \kappa$). 
  Also in $V^\mathbb{S}$, let $\mathbb{P}$ be an iteration, taken with supports of size $<\lambda$, destroying the stationarity of $\mathbb{T}$-fragile subsets of $\lambda$. 
  Then for every $i < \kappa$, $\mathbb{S} \ast \dot{\mathbb{P}} \ast \dot{\mathbb{T}_i}$ has a dense $\lambda$-directed closed subset and, in $V^{\mathbb{S} \ast \dot{\mathbb{T}}_i}$, $\mathbb{P}$ has a dense $\lambda$-directed closed subset.
\end{lem}

\subsection{Indestructible Stationary Reflection}\label{subsec: indestructible stationary reflection}
In this subsection we gather a few theorems of similar flavour which are independent from the other parts of this paper. 

These theorems show that, given large cardinals, one can force simultaneous stationary reflection at many cardinals. The large cardinals that are required depend on the nature of the cardinal at which we force the stationary reflection. These results will be used in Section~\ref{sec:consistency results}, in which our forcing arguments will require that simultaneous stationary reflection is indestructible under sufficiently closed forcing. For notational ease, we thus make the following definition.

\begin{defn}
  Suppose $\lambda$ is a regular, uncountable cardinal, $\kappa \leq \lambda$, and $S \subseteq \lambda$ is stationary. Then $\Refl^*(< \kappa, S)$ is the statement that, whenever $\mathbb{P}$ is a $\lambda$-directed closed forcing poset and $|\mathbb{P}| \leq \lambda$, then $\Vdash_{\mathbb{P}} ``\Refl(< \kappa, \check{S})"$. $\Refl^*(\kappa, S)$ is given the obvious meaning.
\end{defn}

\begin{rem}
Since the trivial forcing is $\lambda$-directed closed, $\Refl^*(< \kappa, S)$ implies $\Refl( < \kappa, S)$.
\end{rem}

The next theorem is well known, but since we require $\Refl^*(\omega_1, S^{\omega_2}_\omega)$ rather than the more standard $\Refl(\omega_1, S^{\omega_2}_\omega)$, we give a detailed proof.
\begin{thm}\label{thm:reflection at omega_2}
The following are equiconsistent over \ZFC.
\begin{enumerate}
\item{There is a weakly compact cardinal.}
\item{$\Refl^*(\omega_1, S^{\omega_2}_\omega)$.}
\end{enumerate}
\end{thm}
\begin{proof}
The equiconsistency of a weakly compact cardinal and the reflection principle $\Refl(2, S^{\omega_2}_\omega)$ is proven in \cite{Magidor1982}. Thus, it will suffice to show that, starting with a weakly compact cardinal, we can force $\Refl^*(\omega_1, S^{\omega_2}_\omega)$. 

Let $\kappa$ be a weakly compact cardinal and assume \GCH. We define an iteration $\langle \mathbb{L}_\alpha, \dot{\mathbb{C}}_\beta \mid \alpha \leq \kappa, \beta<\kappa\rangle$, taken with countable supports. For $\alpha < \kappa$, let $\dot{\mathbb{C}}_\alpha$ be an $\mathbb{L}_\alpha$-name for a two-step iteration, where the first iterand is the lottery sum of all $\alpha$-directed closed forcing notions of size $|\alpha|$ from $V_{\alpha + \omega}^{\mathbb{L}_{\alpha}}$ and the second iterand is $\Col(\omega_1,\alpha)$. Let $\mathbb{L} = \mathbb{L}_\kappa$. 

$\mathbb{L}$ is $\kappa$-c.c.\ and $\sigma$-closed. It is clear that every cardinal $\omega_{1}<\beta<\kappa$ is collapsed, so $\Vdash_{\mathbb{L}}\check{\kappa}=\aleph_{2}$. We claim that $\Refl^*(\omega_1, S^{\omega_2}_\omega)$ holds in $V^{\mathbb{L}}$. To see this, let $\mathbb{P}$ be an $\omega_2$-directed closed forcing of size $\omega_2$ in $V^{\mathbb{L}}$. Let $\{\dot{S_{i}}\mid i<\omega_1\}$ be a collection of names for stationary subsets of $S_\omega^{\omega_2}$ in $V^{\mathbb{L}\ast\mathbb{P}}$. Let \[j\colon\langle V_{\kappa},\mathbb{L}\ast\mathbb{P},\{\dot{S}_{i}\mid i<\omega_{1}\},\in\rangle\to\langle M,j(\mathbb{L})\ast j(\mathbb{P}),\{j(\dot{S}_{i})\mid i<\omega_{1}\},\in\rangle\]
be a weakly compact embedding (we assume, for simplicity, that $\mathbb{P}\subseteq \kappa$).  

Let $G$ be a $V$-generic filter for $\mathbb{L}$, let $H$ be a $V[G]$-generic filter for $\mathbb{P}$, and let $C$ be a $V[G][H]$-generic filter for $\Col(\omega_1, \kappa)$. We now build in $V[G][H][C]$ an $M$-generic filter $K\ast J$ for $j(\mathbb{L})\ast j(\mathbb{P})$ such that $j``G\ast H \subseteq K\ast J$.

Since $\mathbb{L}$ is an iteration with bounded support at $\kappa$, for every $q\in\mathbb{L}$, $j(q)=q$. Thus, we can let the portion of $K$ up to stage $\kappa$ to be equal to $G$. Moreover, since $\mathbb{P}\in M$, we can pick the $\kappa^{\mathrm{th}}$ stage of the iteration $j(\mathbb{L})$ to be $\mathbb{P}\ast \Col(\omega_1,\kappa)$ and take $H\ast C$ to be the portion of $K$ at stage $\kappa$. For the stages strictly between $\kappa$ and $j(\kappa)$, we claim that we can find an $M[G][H][C]$-generic filter in $V[G][H][C]$. This is true because the forcing is $\sigma$-closed, $V[G][H][C]\models|M[G][H][C]|=\omega_1$ and $^{\omega}M[G][H][C]\subseteq M[G][H][C]$. Thus, can we enumerate the dense open sets of the tail of the iteration that lie in $M[G][H][C]$ in a sequence of length $\omega_1$ and build a suitable generic filter. This completes the construction of $K$. 

Finally, since $j(\mathbb{P})$ is $j(\kappa)$-directed closed and \[j``H = H\in M[G][H]\] (since we assumed that $\mathbb{P}\subseteq \kappa$ and therefore its conditions do not move under $j$), there is a condition $p^\star\in j(\mathbb{P})$ such that $p^\star\leq j(p)$ for every $p\in H$. We build in $V[G][H][C]$ an $M[K]$-generic filter, $J$, below this master condition. 

Thus, in $V[G][H][C]$, there is an elementary embedding \[\tilde{j}\colon V_{\kappa}[G][H]\to M[K][J].\] 

Since $\Col(\omega_1,\kappa)$ is $\sigma$-closed, for each $i<\omega_1$, \[V[G][H][C]\models ``S_i\text{ is stationary}."\] In particular, $M[K][J]\models ``S_i$ is stationary$,"$ and therefore, since, for all $i < \omega_1$, $\tilde{j}(S_i) \cap \kappa = S_i$, we have \[M[K][J]\models``j(\{S_i \mid i < \omega_1\})\text{ reflects simultaneously at }\kappa."\] Since $\tilde{j}$ is elementary, $\{S_i \mid i < \omega_1\}$ reflects simultaneously in $V_\kappa[G][H]$.
\end{proof}
We quote the following theorem from \cite[Section 6.3]{Magidor-Cummings-Foreman-Squares}.
\begin{thm} \label{thm:indestructible singular successor reflection}
If the existence of infinitely many supercompact cardinals is consistent, then it is consistent that, for all $n < \omega$, $Refl^*(< \aleph_\omega, S^{\aleph_{\omega + 1}}_{\leq \aleph_n})$ holds.
\end{thm}
The indestructibility of the stationary reflection at $\aleph_{\omega+1}$ is not mentioned explicitly in \cite{Magidor-Cummings-Foreman-Squares}, but it follows easily from the proof of the theorem.
\begin{thm}
If the existence of an inaccessible limit of supercompact cardinals is consistent, then it is consistent that, letting $\kappa$ be the least inaccessible cardinal, $\Refl^*(< \kappa, \kappa)$ holds.
\end{thm}
\begin{proof}
Let $\kappa$ be the least inaccessible limit of supercompact cardinals, and let $\{\mu_{i}\mid i<\kappa\}$ be a continuous, increasing sequence of cardinals cofinal in $\kappa$ such that:
\begin{enumerate}
  \item $\mu_0 = \omega$;
  \item for all limit ordinals $i < \kappa$, $\mu_{i+1} = \mu_i^+$;
  \item for all successor ordinals $i < \kappa$, $\mu_{i+1}$ is supercompact. 
\end{enumerate}
Let $\mathbb{P}$ be the Easton-support iteration of $\Col(\mu_{i+1},<\mu_{i+2})$ for $i<\kappa$. In $V^{\mathbb{P}}$, we have $\mu_i = \aleph_i$ for all $i < \kappa$ and $\kappa$ is the least inaccessible cardinal. We claim that $\Refl^*(<\kappa, \kappa)$ holds in $V^{\mathbb{P}}$. To see this, let $\mathbb{Q}$ be a $\kappa$-directed closed forcing in $V^{\mathbb{P}}$, let $\eta < \kappa$, and let $\{S_\xi\mid \xi<\eta\}$ be a sequence of stationary subsets of $\kappa$ in $V^{\mathbb{P}\ast\mathbb{Q}}$.

Since $\kappa$ is non-Mahlo, by thinning out the stationary sets if necessary, we can assume without loss of generality that there is a successor ordinal $i < \kappa$ such that $\eta < \mu_i$ and, for all $\xi < \eta$, $S_\xi \subseteq S^\kappa_{<\mu_i}$. Let $j\colon V\to M$ be a $\kappa$-supercompact embedding with critical point $\mu_{i+1}$.

$\mathbb{P}$ can be written as $\mathbb{P}_{i}\ast \Col(\mu_{i},<\mu_{i+1})\ast\mathbb{P}^{i+1}$, where $\mathbb{P}_{i}$ is the first $i$ steps in the iteration and $\mathbb{P}^{i+1}$ is $\mu_{i+1}$-directed closed, so we can write $j(\mathbb{P}\ast\mathbb{Q})$ as $\mathbb{P}_{i}\ast \Col(\mu_{i},<j(\mu_{i+1}))\ast j(\mathbb{P}^{i+1})\ast j(\mathbb{Q})$. By standard techniques, if $G$ is $\mathbb{P}$-generic over $V$ and $H$ is $\mathbb{Q}$-generic over $V[G]$, there is a $\mu_i$-closed poset $\mathbb{R}$ in $V[G][H]$ such that, if $I$ is $\mathbb{R}$-generic over $V[G][H]$, then, in $V[G][H][I]$, we can lift $j$ to an elementary embedding $j:V[G][H] \rightarrow M[G][H][I]$. Since $\mu_i$-closed forcing preserves stationary subsets of $S^\kappa_{<\mu_i}$ when $\kappa$ is inaccessible, we have that, for all $\xi < \eta$, $S_\xi$ is stationary in $V[G][H][I]$. Since $j``\kappa \in M[G][H][I]$ and is $(<\mu_i)$-club in $\delta = \sup(j``\kappa)$, we conclude that, for all $\xi < \eta$, $M[G][H]{i} \models ``j(S_\xi) \cap \delta$ is stationary in $\delta."$ In particular, since $j(\{S_\xi \mid \xi < \eta\}) = \{j(S_\xi) \mid \xi < \eta\}$, $M[G][H][I] \models ``j(\{S_\xi \mid \xi < \eta\})$ reflects simultaneously$."$ By elementarity, $\{S_\xi \mid \xi < \eta\}$ reflects simultaneously in $V[G][H]$.
\end{proof}

\section{Consistency results}\label{sec:consistency results}
In this section, we will apply the methods of Section \ref{sec: forcing preliminaries} to show that some of the results from Section \ref{sec: zfc results} are optimal. We start by separating principles of simultaneous reflection of finitely many stationary sets. In what follows, $I[\lambda]$ denotes the approachability ideal 
on $\lambda$. For information on $I[\lambda]$, we direct the reader to \cite{eisworth}.

\begin{thm}
  Let $\mu < \lambda$ be regular cardinals with $\lambda^{<\lambda} = \lambda$, and let $n$ be a natural number. Assume that $\Refl^*(n, S^\lambda_\mu)$ holds and $S^\lambda_\mu \in I[\lambda]$. Then there is a generic extension preserving cardinals and cofinalities in which $\Refl(n, S^\lambda_\mu)$ holds and $\Refl(n+1, S^\lambda_\mu)$ fails.  
\end{thm}
It is interesting to compare this result to the result of Beaudoin in \cite{Beaudoin1991}, where similar separation of variations of $\PFA$ is obtained, using similar method. Let $\PFA^+(n)$ be the assertion that for every proper forcing, $\mathbb{P}$, a collection of $\aleph_1$ many dense subsets of $\mathbb{P}$, $\mathcal{D}$, and $n$ many names for stationary subsets of $\omega_1$, $\dot{S}_i$, there is a filter which meet every member of $\mathcal{D}$ and realizes each $\dot{S}_i$ as a stationary set. In \cite{Beaudoin1991}, Beaudoin shows that for any model of $\PFA^{+}(n)$ and $m  > n$ there is a generic extension in which $\PFA^+(n)$ still holds and there is a collection of $m$ stationary sets that does not reflect together.   
\begin{proof}
  Let $\mathbb{S}$ be the forcing that adds $n+1$ disjoint subsets of $S^\lambda_\mu$ that do not reflect simultaneously using bounded conditions. More precisely, conditions in $\mathbb{S}$ are functions 
  $s: (n+1) \times \gamma^s \rightarrow 2$ satisfying the following conditions, where, for $i < n+1$, $s_i: \gamma^s \rightarrow 2$ is defined by $s_i(\alpha) = s(i, \alpha)$ and, in a slight abuse of notation, we will also think of $s_i$ as the subset of $\gamma^s$ whose characteristic function is $s_i$.
  \begin{enumerate}
    \item{$\gamma^s < \lambda$.}
    \item{For all $i < n+1$, $s_i \subseteq S^\lambda_\mu$.}
    \item{For all $i < j < n+1$, $s_i \cap s_j = \emptyset$.}
    \item{For all ordinals $\beta \leq \gamma^s$ with $\cf(\beta) > \omega$, there is $i < n+1$ such that $s_i \cap \beta$ is not stationary in $\beta$.}
  \end{enumerate}
  If $s,t \in \mathbb{S}$, then $t \leq s$ iff $t \supseteq s$. Standard arguments show that $\mathbb{S}$ is $\lambda$-strategically closed and hence $\lambda$-distributive. Let $G$ be $\mathbb{S}$-generic over $V$, and, in $V[G]$, let $\{S_i \mid i < n+1 \}$ be the generic sets, i.e., for $i < n+1$, $S_i = \bigcup_{s \in G} s_i$. Clearly, $\{S_i \mid i < n+1 \}$ does not reflect simultaneously. Also, a simple genericity argument shows that, for every $i < n+1$, $S_i$ is stationary in $\lambda$. Thus, $\Refl(n+1, S^\lambda_\mu)$ fails in $V[G]$. 
  
  For $i < n+1$, let $\mathbb{T}_i$ be the forcing that adds a club in $\lambda$ disjoint from $S_i$. Conditions of $\mathbb{T}_i$ are closed, bounded subsets of $\lambda$ disjoint from $S_i$, and $\mathbb{T}_i$ is ordered by end-extension. By arguments similar to those found in Section \ref{sec: forcing preliminaries}, in $V$, for each $i < n+1$, $\mathbb{S}\ast \dot{\mathbb{T}}_i$ contains a dense $\lambda$-directed closed subset, namely the set of $(s, \dot{t}) \in \mathbb{S} \ast \dot{\mathbb{T}}_i$ such that:
  \begin{enumerate}
    \item{there is $t \in V$ such that $s \Vdash_{\mathbb{S}}``\dot{t} = \check{t}"$;}
    \item{$\gamma^s = \max(t) + 1$.}
  \end{enumerate}

  Let us work in $V[G]$ and show that $\Refl(n, S^\lambda_\mu)$ holds there. Fix $n$ stationary subsets of $S^\lambda_\mu$, $A_0, \dots, A_{n-1}$. Since $\{S_i \mid i < n+1\}$ are pairwise disjoint, there is $i < n+1$ such that, for all $j < n$, $A_j \setminus S_i$ is stationary in $\lambda$. Otherwise, by the pigeonhole principle, there would be $i < i^\prime < n+1$ and $j < n$ such that $A_j \setminus S_i$ and $A_j\setminus S_{i^\prime}$ are both non-stationary and hence $S_i \cap S_{i^\prime}$ is stationary. Since, in $V$, $\mathbb{S} \ast \dot{\mathbb{T}}_i$ has a dense $\lambda$-directed closed subset and $\Refl^*(n, S^\lambda_\mu)$ holds, if $H$ is $\mathbb{T}_i$-generic over $V[G]$, then $\Refl(n, S^\lambda_\mu)$ holds in $V[G\ast H]$.
  \begin{claim}
    For all $j < n$, $A_j$ is stationary in $V[G\ast H]$.
  \end{claim}

  \begin{proof}
    Work in $V[G]$. Fix $j < n$, $t \in \mathbb{T}_i$, and $\dot{C}$, a $\mathbb{T}_i$-name forced by $t$ to be a club in $\lambda$. Since $S^\lambda_\mu \in I[\lambda]$ in $V$ and $\mathbb{S}$ preserves all cofinalities, $S^\lambda_\mu \in I[\lambda]$ in $V[G]$. In particular, $A_j \setminus S_i \in I[\lambda]$. Thus, letting $\theta$ be a sufficiently large, regular cardinal and $\triangleleft$ a fixed well-ordering of $H(\theta)$, we can find an internally approachable chain $\langle N_\eta \mid \eta < \mu \rangle$ of elementary substructures of $(H(\theta), \in, \triangleleft)$ such that:
    \begin{enumerate}
      \item{$\mathbb{T}_i, t, \dot{C} \in N_0$;}
      \item{for all $\eta < \mu$, $|N_\eta| < \mu$;}
      \item{letting $N = \bigcup_{\eta < \mu} N_\eta$ and $\delta = \sup(N \cap \lambda)$, we have $\delta \in A_j \setminus S_i$.}
    \end{enumerate}
    Since $\mathbb{T}_i$ is $\mu$-closed, it is now straightforward to build a decreasing sequence $\langle t_\xi \mid \xi < \mu \rangle$ 
    of conditions from $\mathbb{T}_i$ satisfying the following.
    \begin{enumerate}
      \item{$t_0 = t$;}
      \item{for all $\xi < \mu$, $t_\xi \in N_{\xi + 1}$;}
      \item{for all $\eta < \mu$, $\sup(N_\eta \cap \lambda) < \max(t_{\eta + 1})$;}
      \item{for all $\eta < \mu$, there is $\alpha > \sup(N_\eta \cap \lambda)$ such that 
	  $t_{\eta + 1} \Vdash ``\check{\alpha} \in \dot{C}."$}
    \end{enumerate}
    Since $\delta \not\in S_i$, we have $t^* := \{\delta\} \cup \bigcup_{\xi < \mu} t_\xi \in \mathbb{T}_i$. Moreover, $t^*$ forces $\delta \in \dot{C}$; in particular, 
    $t^* \Vdash_{\mathbb{T}_i}``\dot{C} \cap \check{A_j} \neq \emptyset."$ Thus, $A_j$ remains stationary in $V[G\ast H]$. 
  \end{proof}
  Since $\Refl(n, S^\lambda_\mu)$ holds in $V[G\ast H]$, we obtain that $\{A_j \mid j < n\}$ reflects simultaneously in $V[G\ast H]$ and hence also in $V[G]$.
\end{proof}

\begin{rem}
  The assumption that $S^\lambda_\mu\in I[\lambda]$ occurs naturally in many cases. For example, for any regular, uncountable cardinal $\lambda$, $S^\lambda_\omega \in I[\lambda]$. Also, if $\lambda$ is strongly inaccessible, then $S^\lambda_\mu \in I[\lambda]$ for all regular $\mu < \lambda$, and if $\lambda = \mu^+$ and $\mu$ is regular then $S^{\lambda}_{<\mu}\in I[\lambda]$.

In addition, starting with infinitely many supercompact cardinals, iterating Levy collapses, and then forcing to shoot a club through the set of approachable points will yield a model in which $\Refl^*(< \kappa, S^\lambda_\mu)$ and $S^\lambda_\mu \in I[\lambda]$ both hold, where $\mu < \kappa$, $\mu$ is regular, $\kappa$ is singular, and $\lambda = \kappa^+$.
\end{rem}

Let us show next that $\Refl^*(< \kappa, S)$ actually implies the indestructibility of $\Refl(< \kappa, S)$ under a wider class of forcings. This fact will be useful in the results to follow. We first recall the following definitions.

\begin{defn}
Suppose $\mathbb{P}$ is a forcing poset, $\lambda$ is a cardinal such that $\lambda^{<\lambda}=\lambda$, and $\theta$ is a sufficiently large regular cardinal. For $M \prec H(\theta)$, let $\mathbb{P}_M = \mathbb{P} \cap M$. 
\begin{enumerate}
\item{Suppose $M \prec H(\theta)$. A condition $p \in \mathbb{P}$ is \emph{strongly $(M,\mathbb{P})$-generic} if $p$ forces that $G_{\mathbb{P}} \cap M$ is a $\mathbb{P}_M$-generic filter over $V$. Equivalently, every dense, open subset of $\mathbb{P}_M$ is pre-dense below $p$ in $\mathbb{P}$.}
\item{Suppose $M \prec H(\theta)$. $M$ is \emph{$\lambda$-suitable} for $\mathbb{P}$ if:
\begin{itemize}
\item{$\mathbb{P} \in M$;}
\item{$|M| = \lambda$;}
\item{$^{<\lambda}M \subseteq M$.}
\end{itemize}}
\item{Let $\lambda$ be a regular cardinal. $\mathbb{P}$ is \emph{strongly $\lambda$-proper} if, whenever $M \prec H(\theta)$ is $\lambda$-suitable for $\mathbb{P}$ and $p \in M \cap \mathbb{P}$, there is $q \leq p$ such that $q$ is strongly $(M, \mathbb{P})$-generic.}
\end{enumerate}
\end{defn}

\begin{thm}
Suppose $\lambda$ is a regular cardinal, $\lambda^{<\lambda} = \lambda$, $S \subseteq \lambda$ is stationary, and $\kappa \leq \lambda$. Then the following are equivalent:
\begin{enumerate}
\item{$\Refl^*(<\kappa, S)$;}
\item{$\Refl(< \kappa, S)$ holds in any forcing extension by a $\lambda$-directed closed, strongly $\lambda$-proper forcing poset.}
\end{enumerate}
\end{thm}

\begin{proof}
If $\mathbb{P}$ is a forcing poset and $|\mathbb{P}| \leq \lambda$, then $\mathbb{P}$ is trivially strongly $\lambda$-proper, so (2) easily implies (1). Thus, assume $\Refl^*(< \kappa, S)$ holds, and let $\mathbb{P}$ be a $\lambda$-directed closed, strongly $\lambda$-proper forcing poset. Suppose for sake of contradiction that there is $p \in \mathbb{P}$, $\mu < \kappa$, and a set of $\mathbb{P}$-names $\dot{\mathcal{S}} = \{\dot{S}_\eta \mid \eta < \mu \}$ such that $p \Vdash_{\mathbb{P}}``\dot{\mathcal{S}}$ is a set of stationary subsets of $\check S$ that does not reflect simultaneously.$"$ We may thus fix a set of $\mathbb{P}$-names $\dot{\mathcal{C}} = \{\dot{C}_\beta \mid \beta \in S^\lambda_{>\omega} \}$ such that, for all $\beta \in S^\lambda_{>\omega}$, $p \Vdash_{\mathbb{P}}``\dot{C}_\beta$ is club in $\beta$ and there is $\eta < \mu$ such that $\dot{C}_\beta \cap \dot{S}_\eta = \emptyset."$

Let $\theta$ be a sufficiently large regular cardinal, let $\vartriangleleft$ be a fixed well-ordering of $H(\theta)$, and let $M \prec (H(\theta), \in, \vartriangleleft)$ be $\lambda$-suitable for $\mathbb{P}$ with $p, S, \dot{\mathcal{S}}, \dot{\mathcal{C}} \in M$. Let $\mathbb{P}_M = \mathbb{P} \cap M$. Since $\mathbb{P}$ is $\lambda$-directed closed and $^{<\lambda}M \subseteq M$, $\mathbb{P}_M$ is a $\lambda$-directed closed forcing of size $\lambda$. Thus, since $\Refl^*(< \kappa, S)$ holds, $\Vdash_{\mathbb{P}_M} ``\Refl(< \kappa, S)."$

For $\eta < \mu$, we may form a $\mathbb{P}_M$-name $\dot{S}_{\eta, M}$ for a subset of $S$ such that, for all $\alpha \in S$ and all $q \in \mathbb{P}_M$, $q \Vdash_{\mathbb{P}_M} ``\check\alpha \in \dot{S}_{\eta, M}"$ iff $q \Vdash_{\mathbb{P}} ``\check\alpha \in \dot{S}_\eta"$. Let $\dot{\mathcal{S}}_M = \{\dot{S}_{\eta, M} \mid \eta < \mu\}$. We may similarly define $\mathbb{P}_M$ names $\dot{C}_{\beta, M}$ for $\beta \in S^\lambda_{>\omega}$.

\begin{claim}
For all $\beta \in S^\lambda_{>\omega}$, $p \Vdash_{\mathbb{P}_M}``\dot{C}_{\beta, M}$ is club in $\check\beta."$
\end{claim} 

\begin{proof}
For $\alpha < \beta$, let $D_\alpha = \{q \in \mathbb{P} \mid$ for some $\alpha^* \in (\alpha, \beta)$, $q\Vdash_{\mathbb{P}}``\check\alpha \in \dot{C}_\beta."\}$ For limit ordinals $\gamma < \beta$, let $E_\gamma = \{q \in \mathbb{P} \mid q \Vdash_{\mathbb{P}}``\check\gamma \in \dot{C}_\beta"$ or, for some $\alpha < \gamma$, $q\Vdash_{\mathbb{P}}``\sup(\dot{C}_\beta \cap \check\gamma) < \check\alpha."\}$. Since $p \Vdash_{\mathbb{P}}``\dot{C}_\beta$ is club in $\check\beta,"$ each $D_\alpha$ and $E_\gamma$ is dense below $p$ in $\mathbb{P}$. In addition, each $D_\alpha$ and $E_\gamma$ is in $M$ and hence, by elementarity, is dense below $p$ in $\mathbb{P}_M$. The claim follows.
\end{proof}

\begin{claim}
For all $\beta \in S^\lambda_{>\omega}$, $p \Vdash_{\mathbb{P}_M}``$There is $\eta < \check\mu$ such that $\dot{C}_{\beta, M} \cap \dot{S}_{\eta, M} = \emptyset."$
\end{claim}

\begin{proof}
  Suppose for sake of contradiction that $\beta \in S^\lambda_{>\omega}$, $q \in \mathbb{P}_M$, $q \leq p$, and $q \Vdash_{\mathbb{P}_M}``$For all $\eta < \check\mu,$ $\dot{C}_{\beta, M} \cap \dot{S}_{\eta, M} \neq \emptyset."$ Since $\mathbb{P}_M$ is $\lambda$-directed closed and $^{<\lambda}M \subseteq M$, we may assume that there is a sequence $\langle \alpha_\eta \mid \eta < \mu \rangle \in M$ such that, for all $\eta < \mu$, $q \Vdash_{\mathbb{P}_M} ``\check{\alpha}_\eta \in \dot{C}_{\beta, M} \cap \dot{S}_{\eta, M}."$ But then, for all $\eta < \mu$, $q \Vdash_{\mathbb{P}}``\check{\alpha}_\eta \in \dot{C}_\beta \cap \dot{S}_\eta,"$ contradicting the assumption that $p \Vdash_{\mathbb{P}}``$There is $\eta < \check\mu$ such that $\dot{C}_\beta \cap \dot{S}_\beta = \emptyset."$
\end{proof}

Therefore, if $p \Vdash_{\mathbb{P}_M} ``\dot{\mathcal{S}}_M$ is a set of stationary sets$,"$ then $p$ forces $\dot{\mathcal{S}}_M$ to be a counterexample to $\Refl(< \kappa, S)$, which is a contradiction. Thus, it must be the case that $p \Vdash_{\mathbb{P}_M} ``$For some $\eta < \check\mu$, $\dot{S}_{\eta, M}$ is non-stationary.$"$ 

Fix $\eta < \mu$ and $q \in \mathbb{P}_M$ such that $q \leq p$ and $q \Vdash_{\mathbb{P}_M}``\dot{S}_{\eta, M}$ is non-stationary$."$ Let $\dot{D}$ be a $\mathbb{P}_M$-name such that $q \Vdash_{\mathbb{P}_M}``\dot{D}$ is club in $\check\lambda$ and $\dot{D} \cap \dot{S}_\eta = \emptyset."$ Find $r \in \mathbb{P}$ such that $r \leq q$ and $r$ is strongly $(M, \mathbb{P})$-generic, and let $G$ be $\mathbb{P}$-generic over $V$ with $r \in G$. 

Then $G_M := G \cap M$ is $\mathbb{P}_M$-generic over $V$ and $V[G_M] \subseteq V[G]$. Note that the interpretations of $\dot{S}_{\eta, M}$ in $V[G_M]$ and of $\dot{S}_\eta$ in $V[G]$ are equal. Call this interpretation $S_\eta$. Since $p \in G$, $S_\eta$ is a stationary subset of $S$ in $V[G]$ and hence also in $V[G_M]$. Let $D$ be the interpretation of $\dot{D}$ in $V[G_M]$. Since $q \in G_M$, $D$ is club in $\lambda$ and $D \cap S_\eta = \emptyset$, contradicting the fact that $S_\eta$ is stationary in $V[G_M]$.
\end{proof}

Let $\lambda^{<\lambda} = \lambda$, let $\mathbb{S} = \mathbb{S}(\lambda, 1)$, and, in $V^{\mathbb{S}}$, let $\mathcal{C}$ be the generic $\square(\lambda)$-sequence added by 
$\mathbb{S}$ and $\mathbb{T} = \mathbb{T}(\mathcal{C})$. In $V^{\mathbb{S}}$, let $\mathbb{P}$ be the iteration of length $2^\lambda$, taken with supports of size $<\lambda$, destroying 
the stationarity of $\mathbb{T}$-fragile subsets of $\lambda$. For each $\alpha < 2^\lambda$, let $\dot{S}_\alpha$ be a $\mathbb{P}_\alpha$-name for the $\mathbb{T}$-fragile subset of $\lambda$ destroyed by the $\alpha^{\mathrm{th}}$ iterand of $\mathbb{P}$, and let $\dot{C}_\alpha$ be a $\mathbb{P}_\alpha \ast \mathbb{T}$-name for a club in $\lambda$ disjoint from $\dot{S}_\alpha$.

\begin{lem} \label{strongProperIteration}
  $\mathbb{S} \ast \dot{\mathbb{P}} \ast \dot{\mathbb{T}}$ is strongly $\lambda$-proper.
\end{lem}

\begin{proof}
  Let $\theta$ be a sufficiently large regular cardinal and let $M \prec H(\theta)$ be $\lambda$-suitable for $\mathbb{S} \ast \dot{\mathbb{P}} \ast \dot{\mathbb{T}}$. We 
  claim that the empty condition is strongly $(M, \mathbb{S} \ast \dot{\mathbb{P}} \ast \dot{\mathbb{T}})$-generic. To see this, let $(s, \dot{p}, \dot{t}) \in 
  \mathbb{S} \ast \dot{\mathbb{P}} \ast \dot{\mathbb{T}}$, and let $D$ be a dense, open subset of $(\mathbb{S} \ast \dot{\mathbb{P}} \ast \dot{\mathbb{T}}) \cap M$. By extending $(s, \dot{p}, \dot{t})$ 
  if necessary, we may 
  suppose that $(s, \dot{p}, \dot{t})$ is in $\mathbb{U}$, the $\lambda$-directed closed subset of $\mathbb{S} \ast \dot{\mathbb{P}} \ast \dot{\mathbb{T}}$ described in the proof of Lemma 
  \ref{lem: square+killing stat+thread is kappa closed}. Let $a \in V$ be such that $s \Vdash_{\mathbb{S}}``\dom(\dot{p}) = \check{a}"$ and, for $\alpha \in a$, let $c_\alpha \in V$ be such that 
  $(s, \dot{p} \restriction \alpha) \Vdash_{\mathbb{S} \ast \dot{\mathbb{P}}_\alpha}``\dot{p}(\alpha) = \check{c}_\alpha"$. Let $t \in V$ be such that $s \Vdash_{\mathbb{S}}``\dot{t} = \check{t}."$

  Let $(s', \dot{p}', \dot{t}') \in (\mathbb{S} \ast \dot{\mathbb{P}} \ast \dot{\mathbb{T}}) \cap M$ satisfy the following requirements.
  \begin{enumerate}
    \item{$s' = s$ and $s' \Vdash_{\mathbb{S}}``\dot{t}' = \check t."$}
    \item{$s' \Vdash ``\dom(p') = \check a \cap M."$}
    \item{For all $\alpha \in a \cap M$, $\dot{p}'(\alpha)$ is forced by $s'$ to have the following property: if $(r, \dot{q}) \in \mathbb{S} \ast \dot{\mathbb{P}}_\alpha$ 
      and $(r, \dot{q}) \Vdash_{\mathbb{S} \ast \dot{\mathbb{P}}_\alpha} ``\check c_\alpha \cap \dot{S}_\alpha = \emptyset,"$ then $(r, \dot{q}) \Vdash_{\mathbb{S} \ast \dot{\mathbb{P}}_\alpha} 
      ``\dot{p}'(\alpha) = \check c_\alpha"$ and, if $(r, \dot{q}) \Vdash_{\mathbb{S} \ast \dot{\mathbb{P}}_\alpha} ``\check c_\alpha \cap \dot{S}_\alpha \neq \emptyset,"$ then 
      $(r, \dot{q}) \Vdash_{\mathbb{S} \ast \dot{\mathbb{P}}_\alpha} ``\dot{p}'(\alpha) = \emptyset."$}
  \end{enumerate}
  Such an $(s', \dot{p}', \dot{t}')$ can be defined, as ${^{<\lambda}}M \subseteq M$. Find $(s'', \dot{p}'', \dot{t}'') \leq (s', \dot{p}', \dot{t}')$ in $D$. 
  It is routine to verify that $(s'', \dot{p}'', \dot{t}'')$ and $(s, \dot{p}, \dot{t})$ are compatible in $\mathbb{S} \ast \dot{\mathbb{P}} \ast \dot{\mathbb{T}}$. This 
  shows that $D$ is pre-dense below the empty condition and hence that the empty condition is strongly $(M, \mathbb{S} \ast \dot{\mathbb{P}} \ast \dot{\mathbb{T}})$-generic.
\end{proof}

Note that, by essentially the same proof, Lemma \ref{strongProperIteration} remains true if $\mathbb{S}$ is any 
forcing of the form $\mathbb{S}(\lambda, < \kappa)$ and $\mathbb{T}$ is the associated 
threading forcing or if $\kappa < \lambda$ is an infinite, regular cardinal, $\mathbb{S} = \mathbb{S}^{\mathrm{ind}}(\lambda, \kappa)$, 
$i < \kappa$, and $\mathbb{T} = \mathbb{T}_i(\mathcal{C})$, where $\mathcal{C}$ is the generic $\square^{\mathrm{ind}}(\lambda, \kappa)$-sequence added by $\mathbb{S}$.

\begin{thm} \label{squareReflectionThm}
  Suppose $\lambda$ is an uncountable, regular cardinal, $\lambda^{<\lambda} = \lambda$, $S \subseteq \lambda$ is stationary, and 
  $\Refl^*(S)$ holds. Then there is a forcing 
  extension preserving all cofinalities and cardinalities $\leq \lambda$ in which $\square(\lambda)$ and $\Refl(S)$ both hold.
\end{thm}

\begin{proof}
  Let $\mathbb{S} = \mathbb{S}(\lambda, 1)$. Let $G$ be $\mathbb{S}$-generic over $V$ and, in $V[G]$, let $\mathcal{C} = \langle C_\alpha \mid \alpha < \lambda \rangle$ 
  be the generically-added $\square(\lambda)$-sequence. Let $\mathbb{T} = \mathbb{T}(\mathcal{C})$, and let $\mathbb{P}$ be an iteration of length $2^\lambda$, 
  taken with supports of size $< \lambda$, destroying the stationarity of all $\mathbb{T}$-fragile 
  subsets of $S$. Let $H$ be $\mathbb{P}$-generic over $V[G]$. We claim that $V[G\ast H]$ is the desired model.

  We first argue that $\square(\lambda)$ holds. In fact, $\mathcal{C}$ remains a $\square(\lambda)$-sequence in $V[G\ast H]$. 
  Work in $V[G]$, and suppose for sake of contradiction that there is $p \in H$ and a $\mathbb{P}$-name $\dot{t}$ 
  such that $p \Vdash_{\mathbb{P}}``\dot{t}$ is a thread through $\check{\mathcal{C}}."$ By Lemma \ref{lem: square+killing stat+thread is kappa closed},
  $\mathbb{P}$ has a dense $\lambda$-directed closed subset in $V[G]^\mathbb{T}$. In particular, $\mathbb{P} \times \mathbb{P}$ 
  is $\lambda$-distributive in $V[G]$, so $\lambda$ remains regular after forcing with $\mathbb{P} \times \mathbb{P}$. 
  Let $H_0 \times H_1$ be $\mathbb{P} \times \mathbb{P}$-generic over $V[G]$, with $(p,p) \in H_0 \times H_1$. For $i < 2$, 
  let $t_i$ be the interpretation of $\dot{t}$ in $V[G \ast H_i]$. By mutual genericity, $t_0 \neq t_1$. In $V[G\ast (H_0 \times H_1)]$,
  find $\alpha \in \acc(t_0) \cap \acc(t_1)$ such that $t_0 \cap \alpha \neq t_1 \cap \alpha$. Then $C_\alpha = t_0 \cap \alpha \neq t_1 \cap \alpha = C_\alpha$, 
  which is a contradiction. Thus, $\mathcal{C}$ remains a $\square(\lambda)$-sequence in $V[G\ast H]$.
 
  Next, we show that $\Refl(S)$ holds. Let $T \in V[G\ast H]$ be a stationary subset of $S$. By our definition of 
  $\mathbb{P}$, $T$ is not $\mathbb{T}$-fragile, so there is $t \in \mathbb{T}$ such that $t \Vdash_{\mathbb{T}}``T$ is stationary$."$ Let $I$ be $\mathbb{T}$-generic 
  over $V[G\ast H]$ with $t \in I$. Since 
  $\mathbb{S} \ast \dot{\mathbb{P}} \ast \dot{\mathbb{T}}$ is strongly $\lambda$-proper and has a dense $\lambda$-directed closed subset, 
  $\Refl(S)$ holds in $V[G\ast H\ast I]$. Thus, $T$ reflects in $V[G\ast H\ast I]$ so, \emph{a fortiori}, $T$ reflects in $V[G\ast H]$ as well.
\end{proof}

\begin{thm}
  Suppose $\lambda$ is an uncountable, regular cardinal, $\lambda^{<\lambda} = \lambda$, $S \subseteq \lambda$ is stationary, and 
  $\Refl^*(2, S)$ holds. Then there is a forcing extension preserving all cofinalities and cardinalities $\leq \lambda$ in which:
  \begin{enumerate}
    \item{$\square(\lambda, 2)$ holds;}
    \item{for every $T \subseteq S$ such that $T$ and $S \setminus T$ are both stationary, $\{T, S \setminus T\}$ reflects simultaneously.}  
  \end{enumerate}
\end{thm}

\begin{proof}
  Let $\mathbb{S} = \mathbb{S}(\lambda, 2)$. Let $G$ be $\mathbb{S}$-generic over $V$ and, in $V[G]$, let $\mathcal{C} = 
  \langle \mathcal{C}_\alpha \mid \alpha < \lambda \rangle$ be the generically added $\square(\lambda, 2)$-sequence. Let $\mathbb{T} = 
  \mathbb{T}(\mathcal{C})$, and let $\mathbb{P}$ be an iteration of length $2^\lambda$, taken with supports of size $< \lambda$, destroying 
  the stationarity of all $\mathbb{T}$-fragile subsets of $S$. Let $H$ be $\mathbb{P}$-generic over $V[G]$. 
  We claim that $V[G\ast H]$ is the desired model.

  We first show that $\mathcal{C}$ remains a $\square(\lambda, 2)$-sequence in $V[G\ast H]$. Work in $V[G]$, and suppose for sake of contradiction that 
  there is $p \in H$ and a $\mathbb{P}$-name $\dot{t}$ such that $p \Vdash_{\mathbb{P}}``\dot{t}$ is a thread through $\mathcal{C}."$ By arguments as 
  in the proof of Theorem \ref{squareReflectionThm}, $\mathbb{P}^3$ is $\lambda$-distributive in $V[G]$. Let $H_0 \times H_1 \times H_2$ be 
  $\mathbb{P}^3$-generic over $V[G]$ with $(p,p,p) \in H_0 \times H_1 \times H_2$. For $i < 3$, let $t_i$ be the interpretation of $\dot{t}$ in 
  $V[G\ast H_i]$. In $V[G \ast (H_0 \times H_1 \times H_2)]$, let $\alpha \in \bigcap_{i < 3} \acc(t_i)$ be such that $t_0 \cap \alpha$, $t_1 \cap \alpha$, and 
  $t_2 \cap \alpha$ are pairwise distinct. Then $\{t_0 \cap \alpha, t_1 \cap \alpha, t_2 \cap \alpha\} \subseteq \mathcal{C}_\alpha$, contradicting 
  the fact that $|\mathcal{C}_\alpha| \leq 2$.

  We next show that (2) holds in $V[G\ast H]$. First note that, as $\mathbb{S} \ast \dot{\mathbb{P}} \ast \dot{\mathbb{T}}^2$ has a dense 
  $\lambda$-directed closed subset, $S$ remains stationary in $V^{\mathbb{S} \ast \dot{\mathbb{P}} \ast \dot{\mathbb{T}}^2}$. 
  In $V[G\ast H]$, let $T \subseteq S$ be such that $T$ and $S \setminus T$ are both stationary. Let $T_0 = T$ and $T_1 = S \setminus T$. We first 
  claim that there is $i < 2$ such that $\Vdash_{\mathbb{T}}``T_i$ is stationary$."$ To see this, suppose to the contrary that there are $t_0, t_1 \in \mathbb{T}$ 
  such that, for $i < 2$, $t_i \Vdash_{\mathbb{T}}``T_i$ is non-stationary$."$ Let $I_0 \times I_1$ be $\mathbb{T} \times \mathbb{T}$-generic over 
  $V[G\ast H]$ with $(t_0, t_1) \in I_0 \times I_1$. Then, in $V[G\ast H \ast (I_0 \times I_1)]$, $T_0$ and $T_1$ are both non-stationary, contradicting the fact 
  that $T_0 \cup T_1 = S$ and $S$ is stationary.

  Without loss of generality, suppose $\Vdash_{\mathbb{T}}``T$ is stationary$."$ Since $S \setminus T$ is not $\mathbb{T}$-fragile, there is 
  $t \in \mathbb{T}$ such that $t \Vdash_{\mathbb{T}}``S \setminus T$ is stationary$."$ Let $I$ be $\mathbb{T}$-generic over $V[G\ast H]$ with 
  $t \in I$. In $V[G\ast H\ast I]$, $T$ and $S \setminus T$ are both stationary and $\Refl(2, S)$ holds. Thus, $\{T, S \setminus T\}$ reflects 
  simultaneously in $V[G\ast H\ast I]$ and hence in $V[G\ast H]$ as well.
\end{proof}

\begin{thm}
  Suppose $\lambda$ is an uncountable, regular cardinal, $\kappa < \lambda$ is an infinite, regular cardinal, $S \subseteq \lambda$ is stationary, 
  and $\Refl^*( < \kappa, S)$ holds. Then there is a forcing extension preserving all cofinalities and cardinalities $\leq \lambda$ 
  in which $\square^{\mathrm{ind}}(\lambda, \kappa)$ and $\Refl(< \kappa, S)$ both hold.
\end{thm}

\begin{proof}
  Let $\mathbb{S} = \mathbb{S}^{\mathrm{ind}}(\lambda, \kappa)$. Let $G$ be $\mathbb{S}$-generic over $V$ and, in $V[G]$, let $\mathcal{C} = 
  \langle C_{\alpha, i} \mid \alpha < \lambda, i(\alpha) \leq i < \kappa \rangle$ be the generically-added $\square^{\mathrm{ind}}(\lambda, \kappa)$-sequence. For $i < \kappa$, let $\mathbb{T}_i = \mathbb{T}_i(\mathcal{C})$, and let $\mathbb{T} = \bigoplus_{i < \kappa} \mathbb{T}_i$. In $V[G]$, let $\mathbb{P}$ be a forcing iteration of 
  length $2^\lambda$, taken with supports of size $< \lambda$, destroying the stationarity of all $\mathbb{T}$-fragile subsets of $S$. Let $H$ be $\mathbb{P}$-generic 
  over $V[G]$. We claim that $V[G\ast H]$ is the desired model.

  We first show that $\mathcal{C}$ remains a $\square^{\mathrm{ind}}(\lambda, \kappa)$-sequence in $V[G\ast H]$. Work in $V[G]$ and suppose to the contrary 
  that there is $p \in H$, $i < \kappa$, and $\dot{D}$ such that $p \Vdash_{\mathbb{P}} ``\dot{D}$ is a club in $\check\lambda$ and, for all $\alpha \in \acc(\dot{D})$, 
  $\dot{D} \cap \alpha = \check{C}_{\alpha, i}."$ By Lemma \ref{indexed_iteration_lemma}, $\mathbb{P} \times \mathbb{P}$ is $\lambda$-distributive in $V[G]$. Let $H_0 \times H_1$ be 
  $\mathbb{P} \times \mathbb{P}$-generic over $V[G]$ with $(p,p) \in H_0 \times H_1$. For $i < 2$, let $D_i$ be the interpretation of $\dot{D}$ in $V[G\ast H_i]$. 
  In $V[G\ast (H_0 \times H_1)]$, find $\alpha \in \acc(D_0) \cap \acc(D_1)$ such that $D_0 \cap \alpha \neq D_1 \cap \alpha$. Then $C_{\alpha, i} = D_0 \cap \alpha 
  \neq D_1 \cap \alpha = C_{\alpha, i}$, which is a contradiction.

  We next show that $\Refl(< \kappa, S)$ holds in $V[G\ast H]$. Thus, let $\mu < \kappa$, and let $\{S_\eta \mid \eta < \mu \}$ be a family of stationary subsets 
  of $S$. For all $\eta < \mu$, $S_\eta$ is not $\mathbb{T}$-fragile, so there is $i_\eta < \kappa$ and $t_\eta = C_{\gamma_\eta, i_\eta} \in \mathbb{T}_{i_\eta}$ such that 
  $t_\eta \Vdash_{\mathbb{T}_{i_\eta}}``\check{S_\eta}$ is stationary$."$ Fix a limit ordinal $\gamma^* < \lambda$ such that $\sup(\{\gamma_\eta \mid \eta < \mu\}) < \gamma^*$.
  Fix $i^* < \kappa$ such that $\sup(\{i_\eta \mid \eta < \mu\}) < i^*$ and, for all $\eta < \mu$, $\gamma_\eta \in \acc(C_{\gamma^*, i^*})$. Recall that, for $i < i^*$, 
  the function $\pi_{i, i^*} : \mathbb{T}_i \rightarrow \mathbb{T}_{i^*}$ sending $C_{\gamma, i} \in \mathbb{T}_i$ to $C_{\gamma, i^*} \in \mathbb{T}_{i^*}$ is a projection. 
  Thus, for all $\eta < \mu$, $\pi_{i_\eta, i^*}(t_\eta) \Vdash_{\mathbb{T}_{i^*}}``\check{S_\eta}$ is stationary$."$ 
  Let $t^* = C_{\gamma^*, i^*}$ and note that $t^* \in \mathbb{T}_{i^*}$ and, for all $\eta < \mu$, $t^* \leq \pi_{i_\eta, i^*}(t_\eta)$. 
  Thus, for all $\eta < \mu$, $t^* \Vdash_{\mathbb{T}_{i^*}} `` \check{S_\eta}$ is stationary$."$ Let $I$ be $\mathbb{T}_{i^*}$-generic over $V[G\ast H]$ with $t^* \in I$. Then, 
  in $V[G\ast H\ast I]$, $\Refl(< \kappa, S)$ holds and $\{S_\eta \mid \eta < \mu\}$ is a collection of stationary subsets of $S$. Thus, $\{S_\eta \mid \eta < \mu \}$ 
  reflects simultaneously in $V[G\ast H\ast I]$ and therefore also in $V[G\ast H]$.
\end{proof}

Examples of results that can be obtained by combining these theorems with the results from Subsection 
\ref{subsec: indestructible stationary reflection} include the following.

\begin{cor} \label{aleph_2_cor}
  Suppose there is a weakly compact cardinal.
  \begin{enumerate}
	\item{There is a forcing extension in which $\square(\aleph_2)$ 
	    and $\Refl(S^{\aleph_2}_{\aleph_0})$ both hold.}
	  \item{There is a forcing extension in which $\square(\aleph_2, 2)$ holds 
	      and, for all stationary $S \subseteq S^{\aleph_2}_{\aleph_0}$ such that 
	      $S^{\aleph_2}_{\aleph_0} \setminus S$ is stationary, $\{S, S^{\aleph_2}_{\aleph_0} \setminus 
	    S\}$ reflects simultaneously.}
	  \item{If $i < 2$, there is a forcing extension in which $\square^{\mathrm{ind}}(\aleph_2, \aleph_i)$ 
	    and $\Refl(< \aleph_i, S^{\aleph_2}_{\aleph_0})$ both hold.}
  \end{enumerate}
\end{cor}

\begin{rem}
  We note that Clause (1) of Corollary~\ref{aleph_2_cor} already follows from the results of \cite{harrington_shelah} 
  and requires only a Mahlo cardinal. In particular, Harrington and Shelah prove in \cite{harrington_shelah} 
  that, starting from a model in which there is a Mahlo cardinal $\kappa$, there is a forcing extension in 
  which $\kappa = \aleph_2$ and $S^{\aleph_2}_{\aleph_0}$ holds. If this forcing is done using a cardinal 
  $\kappa$ that is Mahlo but is not weakly compact in $L$, then $\square(\aleph_2)$ will hold in 
  the forcing extension.
\end{rem}

\begin{cor}
  Suppose there are infinitely many supercompact cardinals.
  \begin{enumerate}
	\item{There is a forcing extension in which $\square(\aleph_{\omega + 1})$ 
	  and $\Refl(\aleph_{\omega + 1})$ both hold.}
	\item{There is a forcing extension in which $\square(\aleph_{\omega + 1}, 2)$ holds 
	  and, for all $n < \omega$ and all stationary $S \subseteq S^{\aleph_{\omega + 1}}_{\aleph_n}$ 
	  such that $S^{\aleph_{\omega+1}}_{\aleph_n} \setminus S$ is stationary, 
	  $\{S, S^{\aleph_{\omega + 1}}_{\aleph_n} \setminus S\}$ reflects simultaneously.}
	\item{If $m < \omega$, there is a forcing extension in which $\square^{\mathrm{ind}}(\aleph_{\omega + 1}, \aleph_m)$ 
	  holds and, for all $n < \omega$, $\Refl(< \aleph_m, S^{\aleph_{\omega + 1}}_{\leq \aleph_n})$ holds.}
  \end{enumerate}
\end{cor}

\section{Open questions}

\begin{question} \label{q1}
Let $\kappa < \lambda$ be uncountable cardinals, with $\lambda$ regular, and suppose $\square(\lambda, < \kappa)$ holds. Must it be the case that $\Refl(< \kappa, S)$ fails for every stationary $S \subseteq \kappa$?
\end{question}

A case of particular interest is the following.

\begin{question} \label{q1.5}
  Can $\square(\omega_2, \omega)$ and $\Refl(\aleph_0, S^{\omega_2}_\omega)$ consistently hold simultaneously? 
\end{question}

\begin{question} \label{q2}
  Let $\kappa < \lambda$ be uncountable cardinals, with $\lambda$ regular, and suppose $\square(\lambda, < \kappa)$ holds. Must there be a full 
  $\square(\lambda, < \kappa)$-sequence?
\end{question}

Note that a positive answer to Question \ref{q2} would imply a positive answer to Question \ref{q1} (and hence a negative answer to Question \ref{q1.5}).

\begin{question} \label{q3}
  Is there an analogue of Theorem \ref{fullReflectionThm} dealing with $\square(\mu^+, < \mu)$ when $\mu$ is singular?
\end{question}

\providecommand{\bysame}{\leavevmode\hbox to3em{\hrulefill}\thinspace}
\providecommand{\MR}{\relax\ifhmode\unskip\space\fi MR }
\providecommand{\MRhref}[2]{%
  \href{http://www.ams.org/mathscinet-getitem?mr=#1}{#2}
}
\providecommand{\href}[2]{#2}

\end{document}